\documentclass{article}

\usepackage{amsmath,amsfonts,calrsfs,fullpage,amssymb,xcolor,verbatim,eucal,yfonts,mathrsfs,titling}
\usepackage{nicefrac}
\usepackage{bbm}
\usepackage{authblk}
\usepackage[bookmarksnumbered, colorlinks, plainpages]{hyperref}
\hypersetup{colorlinks=true,linkcolor=red,}

\newtheorem{thm}{Theorem}[section]
\newtheorem{lemma}[thm]{Lemma}
\newtheorem{pro}[thm]{Proposition}
\newtheorem{cor}[thm]{Corollary}
\newtheorem{defn}[thm]{Definition}

\newtheorem{hyp}[thm]{Hypothesis}

\newtheorem{rmk1}[thm]{Remark}

\numberwithin{equation}{section}

\newenvironment{proof}{{\sc Proof.}}{\hfill {\sc qed}}

\newenvironment{rmk}{\begin{rmk1} \rm }{\end{rmk1}}

\renewcommand{\hat}[1]{\widehat{#1}}
\renewcommand{\tilde}[1]{\widetilde{#1}}

\newcommand{\N}{\mathbb{N}}

\newcommand{\R}{\mathbb{R}}

\newcommand{\E}{\mathscr{E}}

\newcommand{\elle}{\operatorname{L}}

\newcommand{\fcon}{\operatorname{\mathcal{F}C}}

\newcommand{\mres}{\mathbin{\vrule height 1.6ex depth 0pt width
0.13ex\vrule height 0.13ex depth 0pt width 1.3ex}}

\title{Characterization of BV functions on open domains: the Gaussian case and the general case}
\author{Davide Addona\thanks{email: davide.addona@unimib.it}}
\affil{Department of Mathematics and applications\\
University of Milano Bicocca\\
via Cozzi 55, 20125 Milano, Italy}
\author{Giorgio Menegatti\thanks{email:
giorgio.menegatti@unife.it} }  
\author{Michele Miranda jr.\thanks{email: michele.miranda@unife.it}}
\affil{Department of Mathematics and Computer Science\\
University of Ferrara\\
via Machiavelli 30, I-44121 Ferrara, Italy}

\providecommand{\keywords}[1]{{\textit{Keywords}:} #1}

\begin{document}

\maketitle
\abstract{
We provide three different characterizations of the space 
$BV(O,\gamma)$ of the functions of bounded variation with 
respect to a centred non-degenerate Gaussian measure $
\gamma$ on open domains $O$ in Wiener spaces. Throughout 
these different characterizations we deduce a sufficient 
condition for belonging to $BV(O,\gamma)$ by means of the 
Ornstein-Uhlenbeck semigroup and we provide an explicit 
formula for one-dimensional sections of functions of 
bounded variation. Finally, we apply our technique to 
Fomin differentiable probability measures $\nu$ on a 
Hilbert space $X$, inferring a characterization of the 
space $BV(O,\nu)$ of the functions of bounded variation 
with respect to $\nu$ on open domains $O\subseteq X$.

\vspace{0.5cm}


\section{Introduction}
Functions of bounded variation ($BV$ functions in the 
sequel) have had an important role in several classical 
problems of the Calculus of Variations (see 
\cite{AmbFusPal00} for a complete and an in-depth 
dissertation). In one dimension they have been introduced 
in $1881$ in \cite{J81} by Jordan who also pointed out 
the canonical decomposition of $BV$ functions as the 
difference of two increasing functions.

A correct generalization to higher dimensions required 
over than $50$ years and it is due to Fichera and De 
Giorgi, who related $BV$ functions to distributions. In 
\cite{F54} Fichera defined $BV$ functions as those 
functions whose partial derivatives, in the sense of 
distributions, are measures with finite total variation, 
i.e., given a continuous function $u$ and open set 
$\Omega\in \R^n$, $u$ is a function of bounded variation 
if the values
\begin{align*}
T_iu(Q):=\int_{\partial Q}u\nu_i d\mathcal L^{n-1},
\end{align*}
are finite, where $Q\subseteq \Omega$ is a cube with 
sides parallel to the coordinate axes, $i\in \{1,\cdots,n
\}$ and $\nu_i$ is the $i$-th component of the outward 
pointing unit normal to $\partial Q$.

In \cite{D54} De Giorgi showed that functions whose 
distributional derivatives are measures with finite total 
variation can be characterized by means of the behaviour 
near $0$ of the heat semigroup $T_t$. To be more precise, 
he proved that $u\in L^\infty(\R^n)$ is a $BV$ function 
if
\begin{align*}
I(u):=\lim_{t\downarrow0}\int_{\R^n}|\nabla T_tu|dx <+
\infty.
\end{align*}

Further, in \cite{M641,M642} Mario Miranda provided an 
alternative definition of $BV$ functions introducing the 
functional 
\begin{align*}
V(u,\Omega):=\sup\left\{\sum_{i=1}^m\int_\Omega u^i {\rm 
div} (v^i) dx:v\in [C_c(\Omega)]^{mN}, \ \|v\|_\infty\leq 
1\right\},
\end{align*}
for any $u=(u^1,\ldots,u^m)\in \left[L^1_{\rm loc}
(\Omega)\right]^n$ and any open set 
$\Omega\subseteq \R^N$, 
and showing that $u$ is a function of bounded variation 
if and only if $V(u,\Omega)<+\infty$.

The last important characterization of $BV$ functions is 
in terms of smooth functions. Indeed, a function $u\in 
[L^1(\Omega)]^m$ is a $BV$ function if and only if there 
exists a sequence $(u_n)\subseteq[ C^\infty(\Omega)]^m$ 
which converges to $u$ in $L^1$ and whose gradients are 
uniformly bounded in $L^1$.

In infinite dimension, $BV$ functions have been 
introduced by Fukushima and Hino in \cite{Fu00,FuHi01}. 
The first problem which arises in infinite dimension is 
that there does not exist an analogous of the Lebesgue 
measure. Therefore, in \cite{Fu00, FuHi01} the authors 
deal with a Wiener space, i.e., a Banach space endowed 
with a Gaussian measure $\gamma$ and a related 
differential structure characterized by the 
Cameron--Martin space $H$, and they define the space of $
BV$ functions with respect to $\gamma$ relying upon the 
theory of Dirichlet forms.

The first attempt to study $BV$ functions in Wiener 
spaces with tools which are closer to those of Geometric 
Measure Theory in Euclidean setting is 
\cite{AmbMirManPal10}. Here, the authors consider the 
Wiener space $(X,\gamma,H)$ and analyse the connection 
among the distributional notion of vector-valued 
measures, approximations by means of smooth functions and 
the properties of the Ornstein--Uhlenbeck semigroup, 
which in Wiener spaces plays the role of the heat 
semigroup.
One of the main problem in infinite dimension is the loss 
of local compactness of $X$ 
which does not allow to apply the Riesz Theorem on the 
dual of $C_c(X)$ functions.
Further, it is known that the dual space of $C_b(X)$ is 
strictly larger to the space of signed measures on $X$.

The main result of \cite{AmbMirManPal10} is a 
characterization of $BV$ functions in Wiener spaces which 
is completely analogous to the finite dimensional 
situation. 
Namely, \cite[Theorem 4.1]{AmbMirManPal10} states that  a 
function $u\in L({\rm Log}L)^{1/2}(X,\gamma)$ is a 
function of bounded variation, i.e., 
its distributional derivative along the directions of $H$ 
are finite measures if and only if one of the following 
conditions holds true:
\begin{itemize}
\item[(i)] the functional
\begin{align*}
V(u):=\sup\left\{\int_X u\ \!{\rm div}_\gamma Gd\gamma :G
\in \mathcal FC_c^1(X,H):|G(x)|_H\leq 1 \ \forall x\in X
\right\},
\end{align*}
is finite, where $\mathcal FC_c^1(X,H)$ is the space of 
$H$-valued cylindrical functions with ``compact support'' 
(in the sense that its support is an infinite cylinder 
with compact basis).
\item[(ii)] The functional
\begin{align*}
L(u):=\inf\left\{\liminf_{n\rightarrow+\infty}\|
\nabla_Hu_n\|_{L^1(X,\gamma)}:u_n\in \mathbb D^{1,1}(X,
\gamma), \ u_n\stackrel{L^1}\longrightarrow u\right\},
\end{align*}
is finite, where $ \mathbb D^{1,1}(X,\gamma)$ is the 
Sobolev space defined in \cite{Fu00}.
\item[(ii)] The limit
\begin{align*}
\mathcal I(u):=\lim_{t\downarrow 0}
\int_X|\nabla_HT_t u|_Hd\gamma,
\end{align*}
is finite, where $T_t$ is the Ornstein-Uhlenbeck on
the whole space $X$.
\end{itemize}

The aim of this paper is trying to generalize as more as 
possible the results of \cite{AmbMirManPal10} when an 
arbitrary open domain $O\subseteq X$ is considered. 
We say that $f\in L(\log L)^{1/2}(O,\gamma)$ belongs to 
$BV(O,\gamma)$ if there exists a vector measure $\mu\in 
\mathscr M(O,H)$ such that
\begin{align*}
\int_O f{\partial_h^*}\varphi d\gamma=
-\int_O \varphi d\mu_h,
\end{align*}
for any $h$ in the Cameron--Martin space
Here, $\partial_h^*$ denotes, up to the sign, the adjoint 
in $L^2$ of the
partial derivative along $h\in H$, $\mu_h=[\mu,h]_H$ and 
$\varphi\in {\rm Lip}_c(O,H)$ (set of bounded Lipschitz 
functions with bounded support with positive distance 
from $\partial O$). This definition coincides with
\cite[Definition 3.1]{AmbMirManPal10} when $O=X$, and, in 
Theorem~\ref{lauree} we show that 
$f\in L({\rm Log}L)^{1/2}(O,\gamma)$ is a function of 
bounded variation if and only if one of the following 
conditions is satisfied:
\begin{itemize}
\item[(a)] $V_\gamma(f,O)<+\infty$, where
\begin{align*}
V_\gamma(f,O):= 
& \sup \bigg\{\int_Of {\rm div}_\gamma^FGd\gamma: 
F\subseteq QX^* \mbox{ fin. dim.}, \ G
\in{\rm Lip}_c(O,F) ,  
\ |G(x)|_F\leq 1 \ \forall x\in O\bigg\}.
\end{align*}
\item[(b)] $L_\gamma(u,O)<+\infty$, where
\begin{align*}
L_\gamma(f,O):=\inf\left\{\liminf_{n\rightarrow+\infty}\|
\nabla_Hf_n\|_{L^1(O,\gamma)}:\ f_n\in \mathbb D^{1,1}(X,
\gamma), 
\ f_n \stackrel{L^1(O,\gamma)}\longrightarrow f \right\}.
\end{align*}
\end{itemize}
We point out that the equivalence of possible definitions
of $BV$ functions has been considered in the literature
also in more general metric spaces, see for instance
\cite{AmbDiM14Equ}. This characterization
can also be considered in metric measure spaces
with the construction given in \cite{Gig18Non}.
Nevertheless, in such setting when one wants to define
$BV$ functions using the functional $V_\gamma(f,O)$, one
usually requires the existence of a positive measure
that realizes the total variation. In our
setting we do not require a priori the existence
of such a measure, but we prove its existence
relying on the fact that in Wiener spaces
when $O=X$ this was proved in \cite{AmbMirManPal10}.

In finite dimension the definition of the space 
of $BV$ functions naturally extends to general open 
domains $\Omega \subseteq \R^n$; we refer for instance
to \cite{Bogachev2014} where the question
was addressed on the problem of the existence
of the extension operator.
On the other hand, in infinite dimension 
the situation is quite more complicated;
the first issue we deal with is that in Banach spaces the 
distance is, in 
general, not locally smooth. Moreover, compact sets in 
Banach spaces are not enough to approximate open sets 
(for instance, the closed unit ball is not compact in 
Hilbert spaces). Therefore, it is not obvious to find a 
good space of test functions. However, as often happens 
also in finite dimension, the space of Lipschitz 
functions with bounded support is a good compromise, 
since these functions are compatible with the 
differential structure of $X$ related to the Cameron-
Martin space $H$ and the distance is a Lipschitz 
function.
This choice, unfortunately, makes useless the finite 
dimensional approximations by means of conditional 
expectations (see \cite[Corollary 3.5.2]{Bog98}), which 
also in \cite{AmbMirManPal10} are crucial to get the main 
results.

To show that $V_\gamma(f,O)<+\infty$ we take advantage of 
\cite[Theorem 4.1]{AmbMirManPal10}. Indeed, we prove that 
if $V_\gamma(f,O)<+\infty$, then $V(\overline f)<+\infty
$, where $\overline f$ is the null extension of $f$ on 
$O^c$. This means that $\overline f\in BV(X,\gamma)$ and 
that its distributional derivatives along the directions 
of $H$ are finite measure. We conclude by proving that 
these distributional derivatives satisfy the integration 
by parts formula states above.
The fact that $L_\gamma(f,O)<+\infty$ implies that $f\in 
BV(O,\gamma)$ follows from an argument inspired by 
\cite{MS64}.

Finally, we prove that our techniques can also be applied 
in more general situations. We consider the results in 
\cite{DaPLun18}, where $X$ is a Hilbert space, 
$R\in \mathcal L(X)$ and $BV$ functions on $X$ with 
respect to a Fomin differentiable measure $\nu$ along the 
directions $R^*(X)$ are considered. 
The authors prove that a function $u\in L^1(X,\nu)$ 
belongs to $BV(X,\nu)$ if for any $z\in X$ it holds that
\begin{align*}
V_z(u):=\sup\left\{\int_X u(\langle R\nabla \varphi,z
\rangle-v_z\varphi)d\nu:\ \varphi\in C^1_b(X), \ \|
\varphi\|_\infty\leq 1\right\}<+\infty.
\end{align*}
Our arguments can be adapted to this setting and we 
obtain the characterization of $BV$ functions on $O
\subseteq X$ with respect to $\nu$ by means of the 
variation
\begin{align*}
V_z(f,O):=\sup\left\{\int_Xf(\langle RD\varphi,z\rangle-
\varphi{\it v}_z)d\nu:\varphi\in {\rm Lip}_c(O), \ \|
\varphi\|_\infty\leq 1\right\}.
\end{align*}
The unique additional hypothesis that we need is that 
Lipschitz functions are compatible with $\nu$, in a sense 
that we make explicit later.

The paper is organized as follows. In Section 
\ref{pasticcio} we give definitions and preliminary 
results which will be useful in the sequel of the paper. 
We begin with infinite dimensional measure theory, and we 
prove some properties of vector valued measures. Then, we 
introduce the Wiener space $(X,\gamma,H)$, where $X$ is a 
separable Banach space, $\gamma$ is a centred 
non--degenerate Gaussian measure and $H$ is the 
Cameron--Martin space associated to $\gamma$. Later, we 
present the standard construction of Sobolev spaces in 
Wiener setting and some features of the 
Ornstein--Uhlenbeck semigroup on $X$. We conclude by 
listing the main properties of the Orlicz space 
$L({\rm Log}L)^{1/2}(O,\gamma)$ and by defining the space 
of functions of bounded variation $BV(O,\gamma)$, the 
variation $V_\gamma(f,O)$ and $L_\gamma(f,O)$, where 
$O\subseteq X$ is an open domain.

Section \ref{main} is devoted to prove the equivalent 
definitions of $BV(O,\gamma)$, i.e., we show that a 
function $f\in L({\rm Log}L)^{1/2}(O,\gamma)$ belongs to 
$BV(O,\gamma)$ if and only if either $V_\gamma(f,O)<+
\infty $ or $L_\gamma(f,O)<+\infty$.

In Section \ref{further} we collect some important 
consequences of the results in Section \ref{main}. To be 
more precise, let $T_t$ denote the Ornstein--Uhlenbeck 
semigroup on the whole space. We prove that if for 
$f\in L({\rm Log}L)^{1/2}(O,\gamma)$ the quantity
\begin{align*}
\mathscr J(f,O):=\liminf_{t\downarrow0}\int_O| 
\nabla_HT_t(\bar f)|_Hd\gamma,
\end{align*}
where $\bar f$ is the null extension of $f$ outside
$O$, is finite, then $f\in BV(O,\gamma)$. We are not able 
to prove that this condition is also necessary
and indeed this in general is not the case. 
The main problem is that, at the best of our knowledges, 
the study of the Ornstein--Uhlenbeck semigroup on open 
domains is much more complicated. For example, in the 
whole space $T_t$ has an explicit integral representation 
which allows direct computations. In this direction, in 
\cite{LMP15} the authors study $BV$ functions on $X$ 
restricted to an open convex set $\Omega\subseteq X$ in 
terms of the Ornstein--Uhlenbeck semigroup on $\Omega$ 
(see \cite{Ca} for a first analysis of the 
Ornstein--Uhlenbeck and its properties on convex 
domains). However, In this case the convexity of $\Omega$ 
plays an essential role and it is not possible to 
generalize the techniques in \cite{Ca} for a general open 
domain. Further, as in \cite[Proposition 3.103]
{AmbFusPal00} and \cite[Proposition 3.9]{AmbMirManPal10} 
we describe the connections between the one dimensional 
section of $BV$ functions and directional derivatives.

Finally, in Section \ref{general_measure}, $X$ is a 
Hilbert space and, given $R\in\mathcal L(X)$, we consider 
a probability measure $\nu$ which is Fomin differentiable 
along the directions of $R^*(X)$. Starting from the 
results in \cite{DaPLun18}, we provide a characterization 
of $BV$ functions on open domains $O\subseteq X$ by means 
of the variation of a function on $O$ with respect to $
\nu$.

\subsection{Notations}
Let $X$ be a separable Banach space. We denote by $\|
\cdot\|_X$ its norm and by $X^*$ its topological dual, 
i.e., 
the set of bounded linear functionals on $X$. We denote 
by $\langle\cdot,\cdot\rangle$ the duality between $X$ 
and $X^*$. 
Given $x_1^*,\ldots,x_m^*\in X^*$, we denote by $
\pi_{x_1^*,\ldots, x_m^*}:X\longrightarrow \R^m$ the 
bounded linear map 
$\pi_{x_1^*,\ldots, x_m^*}x:=(\langle x,x_1^*\rangle,
\ldots,\langle x, x_m^*\rangle)$. 
If $F={\rm Span}\{x_1^*,\ldots,x_m^*\}\subseteq X^*$, we 
also write $\pi_F$ instead of 
$\pi_{x_1^*,\ldots, x_m^*}$.  
The symbol $\fcon_b^k(X)$ denotes the space of $k$-times 
Fr\'echet differentiable cylindrical functions with 
bounded derivatives up to the order $k$,
that is $u\in \fcon_b^k(X)$ if there exists $v\in 
C_b^k(\R^m)$ ($k$-times continuously differentiable 
functions
with bounded derivatives) such that $u(x)=v(\pi_{x_1^*,
\ldots, x_m^*}x)$ for 
some $m\in\N$ and $x_1^*,\ldots,x_m^*\in X^*$.

The symbol $C_b^1 (X)$ denotes the space of bounded 
functions from $X$ to $\R$ which are  
Fr\'echet differentiable, with bounded Frech{\'e}t 
derivative.

In the same spirit, we say that $E\subseteq X$
is a cylindrical set if there exists $m\in \N$, 
$x^*_1,\ldots, x^*_m\in X^*$ and $B\mathcal(\R^m)$
such that $E=\pi_{x_1^*,\ldots,x_m^*}^{-1}(B)$.

Given a Borel set $E\in \mathcal B(X)$ and an open set $O
\subseteq X$, $O^c$ is the complementary set of $O$, the 
writing $E\Subset O$ means that $E$ is a bounded subset 
of $O$
with ${\rm dist}(E,O^c)>0$, where ${\rm dist}(E,O^c):=
\inf\{\|x-y\|_E:x\in E, \ y\in O^c\}$. Further, for any 
open set $A\subseteq X$ and any $\eta>0$, 
we define $A_{-\eta}:=\{x\in A: d(x,A^c)>\eta\}$. For any 
$x\in X$ and $r>0$ we denote by $B(x,r):=\{y\in X:
\|y-x\|_X<r\}$. If $x=0$ we simply write $B(r)$ instead 
of $B(0,r)$.

If $f$ is a function defined on $O$, we denote by 
$\bar f$ its null extension.

\section{Definitions and preliminary results}
\label{pasticcio}

\subsection{Infinite-dimensional measure theory}
\label{inf_dim_meas_th}

Let $\mathcal B(X)$ be the Borel $\sigma$-field on $X$, 
$Y$ be a Hilbert space with inner product $[\cdot,
\cdot]_Y$. 
Since $X$ is separable, $\mathcal B(X)$ is generated by 
the family of the cylindrical sets 
(see \cite[Fernique Corollary in I.1.2]{VakTarCho87}). 
For any open set $O\subseteq X$, we denote by $\mathscr 
M(O,Y)$ the set of Borel countably additive measures on 
$O$ which take values 
in a Hilbert space $Y$. 
We stress that $\mathscr M(O,Y)\subseteq \mathscr M(X,Y)$. 
Indeed, if $\mu\in \mathscr M(O,Y)$ we extend $\mu$ on $X$ by introducing the measure $\tilde \mu(E)=\mu(E\cap O)$ for any $E\in \mathcal B(X)$. 
For any $\mu \in \mathscr M(O,Y)$ and $y\in Y$ we denote by $\mu_y$ the scalar measure $[\mu,y]_Y$. 
When $Y=\R$ we simply write $\mathscr M(O)$ instead of $\mathscr M(O,Y)$. We recall that the total variation of a measure $\mu\in\mathscr M(O,Y)$ is a
positive finire measure defined for any Borel set 
$ B\in \mathcal B(O)$ by
\begin{align*}
|\mu|(B):=\sup\left\{\sum_{n=1}^\infty |\mu(B_n)|_Y:\ 
B=\bigcup_{n\in \N} B_n \mbox{ and the Borel sets } B_n
\mbox{ are pairwise disjoint}\right\}, 
\end{align*}
We remind that since we are working in a separable Banach
space, any finite measure is a Radon measure.

\noindent
For a given $E\in \mathcal{B}(O)$ and 
$\mu\in \mathscr{M}(O,Y)$, we denote 
by $\mu\mres E\in \mathscr{M}(E,Y)$ 
the restriction of $\mu$ on $E$ defined by
\[
\mu\mres E (B)=\mu (E\cap B),\qquad
\forall B\in \mathcal{B}(E).
\]

\noindent
By using the polar decomposition $\mu=\sigma_\mu|\mu|$, 
where $\sigma_\mu:O\rightarrow Y$ is a $|\mu|$-measurable map which satisfies $|\sigma_\mu|_Y=1$ for $|\mu|$-a.e. $x\in O$, by the inclusion $\mathscr{M}(O,Y)\subseteq
(C_b(O,Y))^*$ we get the equalities
\begin{align*}
|\mu|(O)
= & \sup\left\{\int_O[ \sigma_\mu,\Phi]_Y d|\mu|:\ \Phi\in C_b(O,Y), \ |\Phi(x)|_Y\leq 1 \ \forall x\in O\right\} \\
= & \sup\left\{\int_O[\sigma_\mu,\Phi]_Y d|\mu|:\ \Phi\in \fcon_b^1(O,Y), \ |\Phi(x)|_Y\leq 1 \ \forall x\in O\right\}
\end{align*}
where 
 \begin{align*}
C_b(O,Y) :=&\{u:O\rightarrow Y {\textrm{ continuous and bounded}}\}, \\
\fcon_b^1(O,Y) := &
 \Big\{u:O\rightarrow Y: \\
&\ u(x)=\sum_{i=1}^n v_i(\pi_{x_1^*,\ldots, x_m^*}x)y_i, \ x\in O, \ n,m\in\N, v_i\in C_b^1(\R^m), \ \ y_i\in Y \ 
 i=1,\ldots,n\Big\}.
\end{align*}
If $Y=\R$, we simply write $C_b(O)$ and $\fcon_b^1(O)$. 

We denote by ${\rm Lip}_c(O,Y)$ the set of 
bounded Lipschitz $Y$-valued functions $G:X\rightarrow Y$ such that ${\rm supp} (G)\Subset O$. If $Y=\R$ we simply write ${\rm Lip}_c(O)$. It is clear that 
if $O_1\subseteq O_2$ then $Lip_c(O_1,Y)
\subseteq Lip_c(O_2,Y)\subseteq Lip_c(X,Y)$. 
Finally, ${\rm Lip}_b(X,Y)$ denotes 
the space of $Y$-valued bounded Lipschitz continuous functions on $X$.

\noindent
The space $C^1_c(O)$ and $C^1_c(O,Y)$ are defined in a 
similar way.

\begin{lemma}\label{banana}
Let $O\subseteq X$ be an open set, and let $\mu\in\mathscr M(O,Y)$. Then, for any open set $A\subseteq O$ we have
\begin{align}
|\mu|(A)=\sup\left\{\int_A[ \sigma,G] _Yd|\mu|: G\in {\rm Lip}_c(A,Y), \ |G(x)|_Y\leq 1 \ \forall x\in O\right\},
\label{corsa}
\end{align}
where $\mu =\sigma |\mu|$ is the polar decomposition of $\mu$. In particular, for any $y\in Y$
\begin{align}
|\mu_y|(A)=\sup\left\{\int_AGd \mu_y: G\in {\rm Lip}_c(A), \ |G(x)|\leq 1 \ \forall x\in A\right\}.
\label{corsaro}
\end{align}
\end{lemma}
\begin{proof}
We limit ourselves to show \eqref{corsa}, since from it we easily deduce \eqref{corsaro}.
Clearly, for any $G\in {\rm Lip}_c(A,Y)$ with $|G(x)|_Y\leq 1$ for any $x\in X$, we have $[\sigma,G]_Y\leq1$, hence
\begin{align*}
\int_A[ \sigma,G] _Y\ d|\mu|\leq |\mu|(A).
\end{align*}\\
We have to prove the converse inequality. Let $\varepsilon>0$. Since $|\mu|$ is a Radon measure, there exists 
a compact set $K\subseteq A$ such that $|\mu|(A\setminus K)<\varepsilon$. 
By the properties of vector measures,
there exists $\varphi_\varepsilon\in C_b(A,Y)$ with  $\|\varphi_\varepsilon\|_{L^\infty(O,\gamma)}\leq 1$ and  such that
\begin{align*}
|\mu|(A)\leq \int_A[ \sigma,\varphi_\varepsilon]_Yd|\mu|+\varepsilon\leq \int_K[\sigma,\varphi_\varepsilon]_Yd|\mu|+2\varepsilon. 
\end{align*}
Let $\{y_n:n\in\N\}$ be an orthonormal basis of $Y$. For $n\in\N$ we consider 
$\sigma_n :=\sum_{i=1}^n [\sigma,y_i]_Y y_i$. From the dominated convergence theorem there exists $n_\varepsilon\in\N$ such that $\sigma_{n_\varepsilon}=:\sigma_\varepsilon$ is 
a Borel function with range contained in ${\rm span}\{y_1,\ldots,y_{n_\varepsilon}\}$, $|\sigma_\varepsilon|_Y\leq 1$ everywhere and
\begin{align*}
\|\sigma-\sigma_\varepsilon\|_{L^1(K,|\mu|)}\leq \varepsilon,
\quad |\mu|(A)\leq \int_K[\sigma_\varepsilon,\varphi_\varepsilon]_Yd|\mu| +3\varepsilon.
\end{align*}
Since $\sigma_\varepsilon$ has finite dimensional range, only a finite number of components of $\varphi_\varepsilon$ is involved in the above integral. 
From the Stone-Weierstrass Theorem, there exists $g_\varepsilon\in \fcon_b^1(X,Y)$ with finite dimensional range such that
$\|g_\varepsilon-\varphi_\varepsilon\|_{L^\infty(K)}<\varepsilon$. To conclude, let $\delta:=d(K,A^c)>0$ and let us consider a Lipschitz function 
$\psi$ such that $\psi\equiv 1$ in $K$ and $\psi\equiv 0$ in $\left(\cup_{x\in K}B(x,\delta/2)\right)^c$. Hence, setting
\begin{align*}
F(h):=
\begin{cases}
h, & |h|_Y\leq 1, \\
h/|h|_Y, & |h|_Y\geq 1,
\end{cases}
\end{align*}
for any $h\in Y$, the function $G_\varepsilon:=\psi \cdot(F\circ g_\varepsilon)$ belongs to ${\rm Lip}_c(A,Y)$, $|G_\varepsilon(x)|_Y\leq 1$ for any $x\in Y$ and
$\|G_\varepsilon-\varphi_\varepsilon\|_{L^\infty(K)}\leq 2\varepsilon$, from which it follows that
\begin{align*}
|\mu|(A)
\leq & \int_K[\sigma_\varepsilon,G_\varepsilon]_Y d|\mu|+5\varepsilon
\leq  \int_A[\sigma,G_\varepsilon]_Y d|\mu| +6\varepsilon\\
\leq & \sup\left\{\int_A[ \sigma,G]_Yd|\mu|:G\in{\rm Lip}_c(A,Y), \ |G(x)|_Y\leq 1\ \forall x\in X\right\}+6\varepsilon.
\end{align*}
The arbitrariness of $\varepsilon$ gives the thesis.
\end{proof}

From the proof of Lemma~\ref{banana} we immediately deduce the following result.
\begin{cor}\label{spaghetto}
Let $\mu\in\mathscr M(O,H)$. Then, for any open set $A\subseteq O$ we have 
\begin{align}
|\mu|(A)=\sup\left\{\int_A[ \sigma,G] _Hd|\mu|: F\subseteq QX^* \mbox{ fin. dim.},\ G\in {\rm Lip}_c(A,F), \ |G(x)|_F\leq 1 \ \forall x\in A \right\},
\label{corsa1}
\end{align}
where $\mu=\sigma |\mu|$ is the polar decomposition of $\mu$.
\end{cor}
\begin{rmk}
{We point out that the space $C_c ^1 (A,Y)$
is not in general a good set of test functions.
Indeed, in the previous proof we have used 
Lipschitz cut-off and the existence of
$C^1$ cut-off functions is strictly related to
the separability of $X^*$  (see 
for instance \cite[Theorem 5.3]{Deville}).} 
\end{rmk}

\begin{rmk}
\label{misura bordo}
Let $\mu$ be a signed Radon measure on $X$ and let $O\subseteq X$ be an open set. By the additivity of $\mu$, and since the sets $ \partial O_{-t}$ are pairwise disjoint,
there exist at most countably many $t\in(0,1)$ 
such that $|\mu|(\partial O_{-t})\neq 0$.
\end{rmk}

\subsection{The abstract Wiener space}
We consider a nondegenerate centred Gaussian measure $\gamma$ on $X$, i.e., $\gamma$ is a probability measure such that, for any $x^*\in X^*$ 
the image measure $\gamma\circ (x^*)^{-1}$ is a centred Gaussian measure on $\R$ and its Fourier transform satisfies
\begin{align*}
\widehat\gamma(x^*):=\int_Xe^{i\langle x,x^*\rangle}\gamma(dx)=\exp\left(-\frac12\langle Qx^*,x^*\rangle\right),
\end{align*}
for some nonnegative and symmetric operator $Q\in\mathcal L(X^*,X)$, said the covariance operator. 
The nondegeneracy hypothesis on $\gamma$ means that $Q$ is a positive definite operator, 
that is, $\langle Qx^*,x^*\rangle>0$ for any $x^*\neq0$. 

Moreover, $Q$ is uniquely determined by
\begin{align*}
\langle Qx^*,y^*\rangle=\int_X\langle x,x^*\rangle\langle x,y^*\rangle\gamma(dx).
\end{align*}
The boundedness of $Q$ follows from Fernique's Theorem (see \cite[Theorem 2.8.5]{Bog98}), which states that there exists $\alpha>0$ such that
\begin{align*}
\int_Xe^{\alpha\|x\|^2_X}\gamma(dx)<+\infty.
\end{align*}
Further, it is easy to prove that the function $x\mapsto\langle x,x^*\rangle$ belongs to $L^p(X,\gamma)$, 
for any $p\in[1,+\infty)$. Let us denote by $R^*:X^*\longrightarrow L^2(X,\gamma)$ the embedding of $X^*$ in $L^2(X,\gamma)$, 
and by $\mathscr H$ the closure of $R^*X^*$ in $L^2(X,\gamma)$. $\mathscr H$ is called the {\it reproducing kernel} of $\gamma$ and 
clearly $R^*X^*$ is dense in it. By putting together Fernique's Theorem and \cite[Theorem 2.10.9]{Bog98}  for any $\hat h \in \mathscr H$ there exists an $\alpha> 0$   such that
\begin{align}\label{palo}
\int_Xe^{\alpha(\hat h(x))^2}\gamma(dx)<+\infty.
\end{align}
It is also possible to prove that $Q=RR^*$, where $R:\mathscr H\longrightarrow X$ is the operator defined by 
the Bochner integral
\begin{align*}
R\hat h:=\int_X\hat h(x)x\gamma(dx), \quad \hat h\in\mathscr H,
\end{align*}
and that $R$ is a injective and compact operator. The space $H:=R\mathscr H\subseteq X$ is called the Cameron--Martin space, and it plays a crucial role
in infinite dimensional analysis. $H$ enjoys nice properties: indeed, $H$ is dense subspace of $X$ and it is a Hilbert space endowed with 
the scalar product $[h_1,h_2]_H:=\langle \hat h_1,\hat h_2\rangle_{L^2(X,\gamma)}$, where $\hat h_1,\hat h_2\in \mathscr H$ 
and $h_i:=R\hat h_i$, $i=1,2$. In particular, if $x^*,y^*\in X^*$, then $[Qx^*,Qy^*]_H=\langle Q x^*,y^*\rangle$. We denote by $|\cdot|_H$ the norm in $H$ induced
by $[\cdot,\cdot]_H$. When no confusion is possible, we simply write $|\cdot |$ instead of $|\cdot |_H$. Moreover, the embedding $H\hookrightarrow X$ is compact and $\gamma(H)=0$ if and only if $X$ is infinite dimensional. The importance of $H$ follows from the Cameron-Martin Theorem (see \cite[Theorem 2.4.5]{Bog98}). For any $h\in X$ let us consider the shifted measure $\gamma^h:=\gamma(\cdot-h)$. Then, $\gamma^h$ is 
absolutely continuous with respect to $\gamma$ if and only if $h\in H$. In this case, if we write $h=R\hat h$,
we have
\begin{align*}
\gamma^h(dx):=\exp\left(\hat h(x)-\frac12|h|^2_H\right)\gamma(dx),
\end{align*}
for any $h\in H$.

\noindent
In the sequel, when $h\in H$, we denote by $\hat{h}
\in \mathscr{H}$ the corresponding element in
the reproducing kernel such that $R \hat{h}=h$.
The non degeneracy of the measure implies that
for $h\in QX^*$ there exists a unique element
$x^*\in X^*$ such that $h=Qx^*$.

For $h\in QX^*$, $h=Qx^*$, we denote by 
$X_h^\perp:={\rm Ker}(\pi_{x^*} )$,  by $\gamma_h$  the image of $\gamma$ under $\pi_{x^*}$ (i.e., $\gamma_h=\gamma\circ \pi_{x^*}^{-1}$) and by $\gamma_h^\perp$ the Gaussian measure on $X$ concentrated on $X_h^\perp$ which is the image of $\gamma$ under ${\rm Id}_X-\pi_{x^*}$. 
With this construction we get the decomposition of $\gamma=\gamma_h\otimes \gamma_h^\perp$.

Finally, there exists an orthonormal basis $\{h_n:n\in\N\}$ of $H$ such that $h_n=Qx_n^*$ with $x^*_n\in X^*$ 
for any $n\in\N$ (see \cite[Corollary 3.2.8]{Bog98}). In the following sections,  we denote
by $\pi_m$ the projection $\pi_{x_1^*,\ldots, x_m^*}$.

\subsection{Sobolev spaces and Ornstein-Uhlenbeck semigroup} 
Due to the Cameron-Martin Theorem, the derivatives along the directions of $H$ will be of crucial importance. Therefore, for any $f\in\fcon_b^1(X)$ and any $h\in H$ we define
\begin{align*}
\partial_hf(x):=\lim_{t\rightarrow0}\frac{f(x+th)-f(x)}{t},
\end{align*}
and
\begin{align*}
\partial_h^*f(x):=\partial_hf(x)-f(x)\hat h(x).
\end{align*}
For any $f\in \fcon_b^1(X)$ with $f(x)=\varphi(\pi_{x_1^*,\ldots, x_m^*}x)$, $\varphi\in C_b^k(\R^m)$ and $x_1^*,\ldots,x_m^*\in X^*$, we set
\begin{align*}
\nabla_Hf(x):=\sum_{i=1}^m\partial_{_i}\varphi(\pi_{x_1^*,\ldots, x_m^*}x)Qx_i^*.
\end{align*}
Clearly, it follows that $\partial_h f(x)=[\nabla_Hf(x),h]_H$ for any $h\in H$.
For a function $\Phi\in \fcon_b^1(X,H)$, if we write
\[
\Phi(x)=\sum_{i=1}^m\varphi_i(x)k_i
\] 
with $m\in \N$, $k_1,\ldots,k_m\in H$ and $\varphi_1,\ldots,\varphi_m\in\fcon _b^1(X)$,
then the $\gamma$-divergence $\Phi$ is defined by
\begin{align*}
{\rm div}_\gamma\Phi(x):=\sum_{i=1}^m\partial_{k_i}^*\varphi_i(x).
\end{align*}
The operators $\partial_h^*$ and ${\rm div}_\gamma$ are,
up to the sign, 
the adjoint operators of $\partial_h$ and $\nabla_H$ in $L^2$, respectively, namely, 
\begin{align}
\int_X\partial_h f gd\gamma=-\int_Xf\partial_h^* gd\gamma, \qquad \int_X[\nabla_Hf,\Phi]_Hd\gamma=-\int_Xf{\rm div}_\gamma \Phi d\gamma,
\label{federazione}
\end{align}
for any $f,g\in\fcon_b^1(X)$ and any $\Phi\in\fcon_b^1(X,H)$. Integration by parts formulae \eqref{federazione} imply that
$\nabla_H:\fcon_b^1(X)\rightarrow L^p(X,\gamma;H)$ is a closable operator for any $p\in[1,+\infty)$ (see \cite[Section 5.2]{Bog98}) and
we still denote by $\nabla_H$ its closure. 
In the next definition we follow the notations of \cite{Fu00}.
\begin{defn}
 For $p\in[1,+\infty)$ we define the \emph{Sobolev space } $\mathbb D^{1,p}(X,\gamma)$ as the domain of the closure of $\nabla_H$ in $L^p(X,\gamma;H)$. 
 We denote the closure as $\nabla_H$
 and $\mathbb D^{1,p}(X,\gamma)$
is a Banach space endowed with the norm 
 \begin{align}
    \left\Vert f\right\Vert_{\mathbb D^{1,p}(X,\gamma)}=\left\Vert f\right\Vert_{L^p(X,\gamma)}+\left\Vert \nabla_H f\right\Vert_{L^p(X,\gamma;H)}.
  \end{align}
\end{defn}

Notice that the same space is denoted by $W^{p,1}(X,\gamma)$ in \cite{Bog98}. 

\begin{rmk}
By approximation it is possible to prove that the first equality in \eqref{federazione} holds true for any $f\in \mathbb D^{1,p}(X,\gamma)$, any $g\in \mathbb D^{1,q}(X,\gamma)$,
with $1<p<+\infty$ and $q=p'$ being its conjugate exponent.
Further, the second equality in \eqref{federazione} holds true for any $f\in \mathbb D^{1,p}(X,\gamma)$ and any $\Phi\in {\rm Lip}_b(X,H)$, with $1\leq p<+\infty$ and 
${\rm Lip}_b(X,H)$ has been introduced in Subsection \ref{inf_dim_meas_th} (see \cite[Proposition 5.8.8]{Bog98}).
\label{salama}
\end{rmk}

Let us introduce the Ornstein-Uhlenbeck semigroup $(T_t)_{t\geq0}$ as follows: for any $f\in L^1(X,\gamma)$, we set
\begin{align}
\label{bassano}
T_tf(x):=\int_Xf(e^{-t}x+\sqrt{1-e^{-2t}}y)\gamma(dy), \quad t\geq0.
\end{align}
Let us recall that $(T_t)_{t\geq0}$ is a strongly continuous contraction semigroup on $L^p(X,\gamma)$, for any $p\in[1,+\infty]$ (see \cite[Proposition 2.4]{Shi04}). Moreover, if $f\in L^p(X,\gamma)$ with $p>1$, then $T_tf\in \mathbb D^{k,q}(X,\gamma)$ for any $k\in\N$ and any $q>1$ (see \cite[Proposition 5.4.8]{Bog98}). 
Further, from the definition of $T_t$ and of $\nabla_H$, if $f\in \mathbb D^{1,1}(X,\gamma)$ then $T_tf\in \mathbb D^{1,1}(X,\gamma)$ and 
\begin{align*}
\nabla_HT_tf=e^{-t}T_t\nabla_Hf,
\end{align*}
where the above equality reads componentwise. 

For every $m\in\N$ and $f\in L^1(X,\gamma)$ we introduce the canonical cylindrical approximation $\mathbb E_mf$ of $f$ as the conditional expectation relative to the $\sigma$-algebra generated by $\{\langle x,x_1^*\rangle, \ldots,\langle x, x_m^*\rangle\}$. \cite[Corollary 3.5.2]{Bog98} show that
\begin{align*}
 \mathbb{E}_mf(x):=\int_Xf(\pi_m x+({\rm Id}_X-\pi_m)y)\gamma(dy).
\end{align*}
 
\begin{defn} 
\label{def:e_h}
Let $h\in H$, let $O\subseteq X$ be an open set and let $f:O\rightarrow \R$ be a $\gamma$-measurable function. For every  $y\in X$ we set
\[
O_{y}:=\{t\in\R|y+th\in O\}
\]
and $f_{y}:O_{y}\rightarrow \R$ defined as 
\[
 f_{y}(t):=f(y+th).
\]
We denote by $D{\E}_{h}^{O}$ the set 
of $\gamma-$measurable functions $f$ such that for $\gamma$-a.e. $y\in X$, the function $f_{y}$ on $O_{y}$ has representative $\tilde{f_{y}}$ (i.e. $f_y(t)=\tilde{f}_y(t)$ for a.e. $t\in\R$ with respect to the Lebesgue measure) which is locally absolutely continuous. 
It is clear that for $\gamma$-a.e. $y\in X$ and for  a.e. $t\in O_{y}$, $f'_y(t)$ is well defined and for such a $t$ we have
\begin{align}\label{civetta}
 f_{y}'(t)=f_{y+rh}'(t-r), \quad \forall r>0.
\end{align}
If for such a $y$ and $t$ we put
\[
\partial_{h}f(y+th):=f_{y}'(t),
\]
$\partial_h f$ is well defined $\gamma$-a.e. and, by \eqref{civetta}, it does not depend on $t$.
\end{defn}

From the definition of $D{\E}^O_h$, it follows that any $f\in D{\E}_{h}^{O}$ has a representative
$\tilde{f}$ such that  for $\gamma$-a.e. $y\in X$
the function $\tilde{f_{y}}$ is locally absolutely continuous.

We can now give the definition of Sobolev spaces on
arbitrary open sets; this approach was used 
in the case $p=2$ and on the whole space $X$ 
by \cite{Hino01} and generalized for 
any $p$ and on domains in \cite{Bogachev2014}

\begin{defn}
\label{def:Sobolev_alternative}Given $p\in[1,+\infty]$, we say that
$f\in W^{1,p}(O,\gamma)$ if $f\in L^{p}(O,\gamma)$, $f\in D{\E}_{h}^{O}$ for all $h\in H$, and there
exists $\nabla_{H}f\in L^{p}(O,\gamma;H)$ such that $\partial_{h}f=[\nabla_{H}f,h] _{H}$.\end{defn}

It is a standard argument to prove that
$W^{1,p}(O,\gamma)$ is a Banach space with norm
given by
\[
\left\Vert f\right\Vert _{W^{1,p}(O,\gamma)}=\left\Vert  f\right\Vert _{L^{p}(O,\gamma)}+
\left\Vert \nabla_{H}f \right\Vert _{L^{p}(O,\gamma;H)}.
\]

\begin{rmk} \label{bottiglia}
Let $f\in W^{1,p}(X,\gamma)$. By definition it follows that $f_{|O}\in W^{1,p}(O,\gamma)$. 
\end{rmk}

The proof of the following result can be deduced 
by the result \cite[Proposition 5.4.6]{Bog98}; we repeat
the proof for reader's convenience.

\begin{lemma}\label{foglio}
$W^{1,p}(X,\gamma)=\mathbb D^{1,p}(X,\gamma)$ for any
$p\in [1,\infty)$.
\end{lemma}
\begin{proof}
If $f\in \fcon_b^1(X)$, then $f$ is an element of $W^{1,p}(X,\gamma)$; hence, by the definition of
 $\mathbb D^{1,p}(X,\gamma)$ and the fact that $W^{1,p}(X,\gamma)$ is complete, we have $\mathbb D^{1,p}(X,\gamma)\subseteq W^{1,p}(X,\gamma)$.
 
\noindent 
To prove the converse inclusion, we consider $f\in\mathbb{D}^{1,p}(X,\gamma)$ and we build a sequence
in $\fcon_{b}^{1}(X)$ which converges to $f$. For $n\in\N$ we define
$f_{n}(x)=\mathbb{E}_{n}f(x)$. We consider 
$\gamma_{n}=\gamma\circ \pi_n ^{-1}$, Borel measure on $\R^n$: clearly there exists $g_{n}\in W^{1,p}(\R^{n},\gamma_{n})$
such that $f_{n}(x)=g_{n}(\widehat{h_{1}}(x),\ldots,\widehat{h_{n}}(x))$
with $\nabla g=(\mathbb{E}_{n}(\partial_{h_{1}}f),\ldots,\mathbb{E}_{n}(\partial_{h_{n}}f))$.
Each $g_{n}$ can be approximated by a sequence $g_{n,m}\in C_{b}^{1}(\R^{n})$.
So we can find a sequence $\{f_{k}\}_{k\in\N}$ in $\fcon_{b}^{1}(X)$
where $f_{k}(x)=g_{n_{k},m_{k}}(\widehat{h_{1}}(x),\ldots,\widehat{h_{n_k}}(x))$,
and $f_{k}$ converges to $f$ in $\mathbb{D}^{1,p}(X,\gamma)$.
\end{proof}

We close this section with the following remark.

\begin{rmk}\label{lugaresi}
If $f\in W^{1,1}(O,\gamma)$ and $g\in {\rm Lip}_c(O)$,
then the integration by parts formula holds
\[
\int_O f \partial^*_h g d\gamma=
-\int_O \partial_h f gd\gamma,
\qquad \forall h\in H.
\]
\end{rmk}

\subsection{The Orlicz spaces}\label{Orlicz}
{We use the concepts of Orlicz space (see \cite{RaoRen91}); in particular we recall two particular examples of Orlicz spaces,
$ L({\rm log}L)^{1/2}(O,\gamma)$ and $ L^{\Psi}(O,\gamma)$.}
Let $O\subseteq X$ be an open set. We introduce the spaces $L({\rm log}L)^{1/2}(O,\gamma)$ and $L^{\Psi}(O,\gamma)$ as follows: let
\begin{align*}
A_{1/2}(x):=\int_0^x(\ln(1+t))^{1/2}dt, \quad x\geq 0,
\end{align*}
and let $\Psi$ be its complementary function, namely,
\begin{align*}
\Psi(y):=\int_0^y(A_{1/2}'(t))^{-1}dt=\int_0^y\left(\exp\left(t^2\right)-1\right)dt, \quad y\geq 0.
\end{align*}
We define the spaces
\begin{align*}
L({\rm log}L)^{1/2}(O,\gamma) & :=\{f\in L^1(O,\gamma):A_{1/2}(|f|)\in L^1(O,\gamma)\}, \\
L^{\Psi}(O,\gamma) & :=\{g\in L^1(O,\gamma):{\textrm{there exists $c>0$ such that }}\Psi(c|g|)\in L^1(O,\gamma)\}.\notag
\end{align*}
We stress that, with the notations of \cite{RaoRen91}, $L({\rm log}L)^{1/2}(O,\gamma)=\tilde {\mathcal L}_{\Phi}(\gamma)$}. 
Since the function $\Phi(t):=(\log(1+t))^{1/2}\in\Delta_2$ 
(see \cite[Definition 1, Section 2.3]{RaoRen91}), from \cite[Theorem 2(ii), Section 3.1]{RaoRen91} it follows that $L({\rm log}L)^{1/2}(O,\gamma)$ is 
a vector space. Following the notations of \cite[Section 3.2]{RaoRen91} we consider the Luxemburg norm
\begin{align*}
N_{\Phi}(f)=\|f\|_{L({\rm log}L)^{1/2}(O,\gamma)}
 & :=\inf\left\{t>0:\int_OA_{1/2}(|f|/t)d\gamma\leq 1\right\}
 \end{align*}
and the space
\begin{align*}
 \tilde L_{\Phi}(\gamma):=(L({\rm log}L)^{1/2}(O,\gamma),\|\cdot\|_{L({\rm log}L)^{1/2}(O,\gamma)})
 =(\tilde{\mathcal L}_{\Phi}(\gamma), N_{\Phi}).
\end{align*} 
Further,
\cite[Corollary 4, Section 3.4]{RaoRen91}
gives $\tilde L_{\Phi}(\gamma)$ is equivalent to the spaces $M_{\Phi}$ and $L_{\Phi}(\gamma)$,
defined in \cite[Definition 2, Section 3.4]{RaoRen91} and \cite[Definition 5, Section 3.1]{RaoRen91}. Hence, \cite[Theorem 10, Section 3.3]{RaoRen91} implies that $L({\rm log}L)^{1/2}(O,\gamma)$ is a Banach space. 
Moreover, similar arguments give that $L^{\Psi}$ is a Banach space endowed with the norm
\begin{align*}
 \|g\|_{L^{\Psi}(O,\gamma)}
 &:=\inf\left\{t>0:\int_O{\Psi}(|g|/t)d\gamma\leq 1\right\}.
\end{align*}
\begin{rmk}
We notice that if two measurable functions 
$g_1,g_2:O\to \R$ have the same image measure, then
\[
\|g_1\|_{L^{\Psi}(O,\gamma)}=
\|g_2\|_{L^{\Psi}(O,\gamma)}.
\]
This simply follows by the identity
\[
\int_O{\Psi}(|g_1|/t)d\gamma
=
\int_\R{\Psi}(|s|/t)(\gamma\circ g_1^{-1})(ds)
=
\int_O{\Psi}(|g_2|/t)d\gamma.
\]
\end{rmk}
We conclude this part with two important results on Orlicz spaces. The former, which is \cite[Proposition 1, Section 3.3]{RaoRen91} and the Remark below therein, 
is a sort of H\"older inequality for complementary Orlicz spaces. The latter is a dominated convergence theorem in Orlicz spaces, and it is \cite[Theorem 14, Section 3.4]{RaoRen91} rewritten in our situation and using our notations.
\begin{pro}\label{prenotazione}
If $f\in L({\rm log}L)^{1/2}(O,\gamma)$ and $g\in L^{\Psi}(O,\gamma)$ then $fg\in L^1(O,\gamma)$ and
\begin{align*}
\|fg\|_{L^1(O,\gamma)}\leq 2 \|f\|_{L({\rm log}L)^{1/2}(O,\gamma)} \|g\|_{L^{\Psi}(O,\gamma)}.
\end{align*}
\end{pro}

\begin{thm}\label{dottorato}
Let $\varphi\in L({\rm log}L)^{1/2}(O,\gamma)$ and let $(\varphi_n)_{n\in\N}$ be a sequence of measurable functions such that $\varphi_n(x)\rightarrow \tilde \varphi(x)$ and $|\varphi_n(x)|\leq |\varphi(x)|$ for $\gamma$-a.e. $x\in O$. Then, $\varphi_n,\tilde \varphi\in L({\rm Log}L)^{1/2}(O,\gamma)$ and
\begin{align*}
\|\varphi_n-\tilde \varphi\|_{L({\rm log}L)^{1/2}(O,\gamma)}\rightarrow 0, \quad n\rightarrow+\infty.
\end{align*}
\end{thm}
Finally, thanks to \eqref{palo}, it is easy to see that the function $x\mapsto |\ell (x)|\in L^{\Psi}(O,\gamma)$, 
and therefore the integral
$\int_{O}f \ell d\gamma$
is well  defined for any $f\in L({\rm log}L)^{1/2}(O,\gamma)$ and $\ell\in \mathscr{H}$.

\subsection{\texorpdfstring{$BV$}{} functions on open domains in infinite dimension}
We begin this subsection by providing the definition of function of bounded variation on an open domain of $X$.
\begin{defn}\label{polenta}
Let $O\subseteq X$ be an open set. We say that $f\in L({\rm log}L)^{1/2}(O,\gamma)$ is a function of bounded variation in $O$, 
and we write $f\in BV(O,\gamma)$, if there exists an $H$-valued measure $\mu_f\in \mathscr M(O,H)$ such that
\begin{align}
\label{pioggia}
\int_O\partial_h^*Gfd\gamma=-\int_OGd[\mu_f,h] _H,
\end{align}
for any $G\in {\rm Lip}_c(O)$ and any $h\in H$. We write $D_\gamma^Of:=\mu_f$ , we call it weak gradient of $f$ and we set $D_\gamma^{O,h}f:=[\mu_f,h]_H$. {As a consequence of Corollary \ref{spaghetto}, if the weak gradient of $f$ there exists then it is unique}. Finally, if $E\in\mathcal B(O)$, $u=\chi_E$ and $u\in BV(O,\gamma)$, then we say that $E$ has finite perimeter in $O$.
\end{defn}

\begin{rmk}
We point out that the requirement on the validity
of \eqref{pioggia} can be equivalently required only
for $h\in QX^*$. Indeed if \eqref{pioggia} holds
for any $h\in QX^*$, we can pass to any $h\in H$
by considering a sequence $(h_j)_{j\in \N}\subseteq
QX^*$ converging to $h$ in $H$. Then 
if $\sigma_j=|h_j-h|_H$ and $k\in H$ is such that
$|k|_H=1$, the functions $\hat{h}_j-\hat{h}$ and
$\sigma_j \hat{k}$ have the same image measure (see e.g \cite[Lemma 2.2.8]{Bog98}). 
By \eqref{palo} $\hat{h}_j-\hat{h}\in L^\Psi (O,\gamma)$ 
and
\[
\| \hat{h}_j-\hat{h}\|_{L^\Psi (O,\gamma)}
=\sigma_j \|\hat{k}\|_{L^\Psi (O,\gamma)}
\]
so $\hat{h}_j\to \hat{h}$ in $L^\Psi (O,\gamma)$.
Thanks to Proposition \ref{prenotazione} we can then
pass to the limit in \eqref{pioggia}.
\end{rmk}

In the next Lemma we state that the Definition \ref{polenta} is equivalent to \cite[Def. 3.1]{AmbMirManPal10}.
\begin{lemma}
\label{tavolo}
If $O=X$, then in Definition \ref{polenta} the space ${\rm Lip}_{\textcolor{red}{c}}(X)$  can be replaced by $\fcon_b^1(X)$. 
\end{lemma}

\begin{rmk}
We stress that there is no inclusion between ${\rm Lip}_c(X)$ and $\fcon_b^1(X)$. Indeed, functions in $\fcon_b^1(X)$ should be smoother than those in ${\rm Lip}_c(X)$. Further, we have no condition on the support of functions in $\fcon_b^1(X)$, and so $\fcon _b^1(X)\not\subseteq {\rm Lip}_c(X)$.
\end{rmk}

\begin{proof} [Proof of Lemma \ref{tavolo}]
Let $f\in L({\rm log}L)^{1/2}(X,\gamma)$ satisfy \eqref{pioggia} for every $\tilde G\in {\rm Lip}_c(X)$, and let $G\in \fcon_b^1(X)$. 
Clearly, $G$ can be approximated by a sequence $(G_n)_{n\in\N}$ of functions in ${\rm Lip}_c(X)$ such that $G_n\rightarrow G$ and 
$\nabla_H G_n\rightarrow \nabla_H G$ $\gamma$-a.e. in $X$ and $G_n$, $\nabla_H G_n$ are uniformly bounded: 
it suffices to consider a sequence $(\theta_n)$ of Lipschitz functions such that $\theta_{n|B(n)} \equiv 1$ and $\theta_{n|X\backslash B(n+1)} \equiv 0$ for every $n\in\N$, and define $G_n =\theta_n G$. For any $h\in QX^*$ we have
\begin{align}
\label{fulmini}
\int_X f\partial_h^*G_n\ d\gamma=-\int_X G_n\ d[ \mu_f,h]_H,
\end{align}
and by the dominated convergence theorem the right-hand side of \eqref{fulmini} converges to 
\[
-\int_XG \ d[ \mu_f,h]_H,\qquad
\mbox{as } n\rightarrow +\infty. 
\]
From the definition of $\partial_h^*$, we can split the left-hand side of \eqref{fulmini} as
\begin{align*}
\int_Xf\partial_h^*G_nd\gamma=\int_Xf\partial_hG_nd\gamma-\int_XfG_n\hat hd\gamma=:I_1^n+I_2^n.
\end{align*}
Again, by the dominated convergence theorem we infer that $I_1^n\rightarrow \int_Xf\partial_hGd\gamma$ as $n\rightarrow+\infty$. 
As far as $I_2^n$ is concerned, we can apply Theorem \ref{dottorato}, with $\varphi:=f \|G\|_\infty$, $\tilde\varphi:=fG$ and $\varphi_n:=fG_n$ 
for any $n\in\N$, so, 
$f(G_n-G)$ goes to $0$ in $L({\rm log}L)^{1/2}(X,\gamma)$. Hence, by Proposition \ref{prenotazione} and by $\hat h \in L^{\boldsymbol \Psi}(X,\gamma)$
we have $I_2^n\rightarrow \int_XfG\hat hd\gamma$, and therefore (\ref{pioggia}) holds true for any $G\in\fcon_b^1(X)$.
\vspace{2mm}

To prove the converse implication, let $f\in L({\rm log}L)^{1/2}(X,\gamma)$, and assume that (\ref{pioggia}) is satisfied for every $\tilde{G}\in \fcon _b^1(X)$.
We claim that every $G\in {\rm Lip}_c(X)$ can be approximated by a sequence $G_n$ of functions in $\fcon_b^1(X)$ such that $G_n\rightarrow G$ and
$\nabla_H G_n\rightarrow \nabla_H G$ $\gamma$-a.e. in $X$ and $G_n$, $\nabla_H G_n$ are uniformly bounded. If the claim is true, we can argue as above to conclude. Hence, it remains to prove the claim. 

Let $G\in {\rm Lip}_c(X)$, and for any $n\in\N$ let $\tilde G_n=\mathbb{E}_nG$. By \cite[Proposition 5.4.5]{Bog98}, $\tilde G_n$ converges 
to $G$ in $\mathbb D^{1,2}(X,\gamma)$ and 
$\tilde G_n$ and $\nabla_H \tilde G_n$  are uniformly bounded. Moreover, $ \tilde G_n$ is a cylindrical function and 
therefore there exists $v_n\in {\rm Lip}_b(\R^n )$ such that $ \tilde G_n=v_n\circ \pi_n$, by identifying $H_n:={\rm span}\{h_1,\ldots,h_n\}$ with $\R^n$. For every $n\in \N$, $v_n$ can be approximated by a sequence $(v_{m,n})_{m\in\N}$ 
of convolutions of $v_n$ with  a sequence of standard mollifiers $\phi_m$ in $\R^n$, and we define $G_{m,n}:=v_{m,n}\circ \pi_n$. 
Easy computations reveal that $G_{m,n}\rightarrow G_n$ in $\mathbb D^{1,2}(X,\gamma)$ 
(see e.g.  \cite[Lemma 3.2]{HINO03}). From the definition, $G_{m,n}\in\fcon_b^1(X)$ and $G_{m,n}$ and $\nabla_H G_{m,n}$ are uniformly bounded 
with respect to $m,n\in\N$.
Now, with a diagonal argument, we find a sequence $G_n\in\fcon_b^1(X)$ which converges to $G$ in  $\mathbb D^{1,2}(X,\gamma)$ and such that 
$G_n$, $\nabla_H G_n$ are uniformly bounded. In particular, up to a subsequence, both $G_n$ and $\nabla_H G_n$ converge to 
$G$ and $\nabla_H G$ $\gamma$-a.e., respectively.
\end{proof}

\begin{rmk}
\label{equality}
{If $O_1 \subseteq O_2$  are  open sets and $f\in BV(O_2,\gamma)$ then $f_{|O_1}\in BV(O_1,\gamma)$.
If $f_1$, $f_2\in BV(X,\gamma)$  and  
$f_{1|O}=f_{2|O}$ then clearly 
\[
D_\gamma^O (f_{1|O}) 
=D_\gamma^O (f_{2|O})=D_\gamma^X f_1\mres O
=D_\gamma^X f_2\mres O.
\] 
We stress that Lemma \ref{tavolo} implies that our definition of $BV(X,\gamma)$ is 
coherent with that used in literature.}
\end{rmk}

\begin{rmk}
\label{pendolare}
As in  \cite[Section 3]{AmbMirManPal10}, we may slightly modify the requirement in Definition \ref{polenta}. Let $\{h_j:j\in\N\}\subseteq QX^*$ be an orthonormal basis of $H$ and 
let $\partial_j^*:=\partial_{h_j}^*$ for any $j\in\N$. We say that $f\in BV(O,\gamma)$ if there exists a family $\{\mu_j\}_{j\in\N}$ of real valued measures such that 
$\sigma:=\sup_{j\in\N}|(\mu_1,\ldots,\mu_j)|(O)<+\infty$ and
\begin{align*}
\int_O u\partial_j^* G\ d\gamma=-\int_OG\ d\mu_j, \quad G\in{\rm Lip}_c(O), \quad j\in\N.
\end{align*}
The measure $\mu:=\sum_{j\in\N}\mu_jh_j$ is well defined and belongs to $\mathscr{M}(O,H)$. 
It is enough to consider the density $f_j$ of $\mu_j$ with respect to $|\mu|$ for any $j\in\N$. Hence, $\sum_{j\in\N}f_j^2\leq 1$ for $\sigma$-a.e. and 
$\mu=\sum_{j\in\N}f_jh_j\sigma$. 
\end{rmk}

Clearly, the restriction of a Sobolev function in $X$ to $O$ is a function of bounded variation.

\begin{lemma}\label{macchiato}
Let $O$ be an open subset of $X$. If $f\in \mathbb D^{1,1}(X,\gamma)$ then $f_{|O}\in  BV(O,\gamma)$ and $D_\gamma^Of=\nabla_Hf \gamma\mres O$.
\end{lemma}

\begin{proof}
From \cite[Proposition 3.2]{FuHi01} it is well known that $ \mathbb D^{1,1}(X,\gamma)$ is continuously embedded into $L({\rm log}L)^{1/2}(X,\gamma)$, 
and so $f\in L({\rm log}L)^{1/2}(O,\gamma)$.
Since ${\rm Lip}_c(O)\subseteq \mathbb D^{1,q}(X,\gamma)$ for any $q\in[1,+\infty]$,
applying Lemma \ref{salama} we have 
\begin{align*}
\int_Of\partial_h^*Gd\gamma
= \int_Xf\partial_h^*Gd\gamma=-\int_X\partial_h fGd\gamma
=-\int_O\partial_hfGd\gamma,
\end{align*}
for any $G\in {\rm Lip}_c(O)$ and any $h\in H$. In particular, for any $h\in H$ we have
\begin{align*}
[D_\gamma^Of,h]_H=\partial_hf \gamma\mres O=[\nabla_H f,h]_H\gamma\mres O,
\end{align*}
and equality $D_\gamma^Of=\nabla_Hf \gamma\mres O$
follows.
\end{proof}

If we consider the Ornstein-Uhlenbeck semigroup defined in \eqref{bassano}, then
\begin{align*}
\int_O (T_t\bar f)({\rm div}_\gamma G) d\gamma=e^{-t}\int_O f\ \!\left({\rm div}_\gamma(T_tG)\right)d\gamma,
\end{align*}
for any $f\in L({\rm log}L)^{1/2}(O,\gamma)$ 
and any $G\in{\rm Lip}_c(O,H)$ (where $T_tG$ is calculated componentwise). 
Indeed, it has been proved in \cite[Section 2.4]{AmbMirManPal10} that
\begin{align*}
\int_X (T_tu) ({\rm div}_\gamma \Phi) d\gamma=e^{-t}\int_X u\ \!\left({\rm div}_\gamma(T_t\Phi)\right)d\gamma,
\end{align*}
for any $u\in L({\rm log}L)^{1/2}(X,\gamma)$ and $\Phi\in \fcon_b^1(X,H)$ ($T_t \Phi$ again calculated componentwise), and by an approximation argument it also holds for any $\Phi\in {\rm Lip}_b(X,H)$. Therefore, if we consider $f\in L({\rm log}L)^{1/2}(O,\gamma)$ we have
\begin{align*}
\int_O (T_t\bar f)({\rm div}_\gamma G) d\gamma
= \int_X (T_t \bar f) ({\rm div}_\gamma G) d\gamma
=e^{-t}\int_X \bar f{\rm div}_\gamma(T_t  G)d\gamma
=e^{-t}\int_O f{\rm div}_\gamma(T_t G)d\gamma,
\end{align*}
for every $G\in{\rm Lip}_c(O,H)$.

The variation of integrable functions plays a crucial role in the setting of $BV$ functions. Indeed, both in finite dimension (see \cite[Definition 3.4 \& Proposition 3.6]{AmbFusPal00}) and in Wiener spaces, when the whole space is considered (see \cite[Definition 3.8 \& Theorem 4.1]{AmbMirManPal10}), it is possible to characterize functions of bounded variation by means of their variation. We introduce this concept also in our context.

\begin{defn}\label{variation}
For any open set $O\subseteq X$ and any $f\in L({\rm log}L)^{1/2}(O,\gamma)$, we define the variation of $f$ in $O$ by
\begin{align}
V_\gamma(f,O):= 
& \sup \bigg\{\int_Of {\rm div}_\gamma^FGd\gamma:
\ F\subseteq QX^* \mbox{ fin. dim.}, \ G\in{\rm Lip}_c(O,F) ,  
\ |G(x)|_F\leq 1 \ \forall x\in O\bigg\}, \label{vista}
\end{align}
where for $F={\rm span}\{k_1,\ldots,k_m\}$ for some $k_1,\ldots,k_m\in QX^*$ and
$G\in{\rm Lip}_c(O,F)$, we define ${\rm div}_\gamma^FG:=\sum_{i=1}^m\partial_{k_i}^* G_i(x)$ with $G(x):=\sum_{i=1}^mG_i(x)k_i$ and 
$|G(x)|_F^2:=\sum_{i=1}^m|G_i(x)|^2$.
Since $f\in L({\rm log}L)^{1/2}(O,\gamma)$, the integral term in \eqref{vista} is well defined for any $F\subseteq QX^*$ of finite dimension.

If $O=X$, we denote $V_\gamma(f,X)$ by $V_\gamma(f)$;
it is not hard to see that this definition coincides with \cite[Definition 3.8]{AmbMirManPal10}, by arguing as in Lemma \ref{tavolo}. 
\end{defn}

Under our assumptions, $G$ has bounded support and $F\subseteq QX^*$, hence ${\rm div}_\gamma^FG$ is bounded. Therefore, it follows that $V_\gamma(f,O)$ is lower semicontinuous with respect to the $L^1$ convergence of $f$.

For any $f\in L(\log L)^{1/2}(O,\gamma)$ we introduce the functional
\begin{align}
\label{fattoria}
L_\gamma(f,O):=\inf\left\{\liminf_{n\rightarrow+\infty}
\int_O |\nabla_Hf_n|_Hd\gamma:\ f_n\in W^{1,1}(O,\gamma), 
\ f_n \stackrel{L^1(O,\gamma)}\longrightarrow f \right\}.
\end{align}

The variation of a function $f$ along a subspace $F$ of $H$ generated by a finite number of elements of $QX^*$ deserves a particular attention. 
Let $h_i,\ldots,h_k\in QX^*$ be orthonormal elements
of $H$ and let $F={\rm span}\{h_{1},\ldots, h_{k}\}$. We define the variation of $f\in L({\rm log}L)^{1/2}(O,\gamma)$ 
in $O$ along $F$ by
\begin{align}
V_\gamma^{F}(f,O):=\sup\left\{\sum_{j=1}^k\int_Of\partial_{h_{j}}^* G_jd\gamma:G_j\in {\rm Lip}_c(O),\ j=1,\ldots,k, 
 \ \sum_{j=1}^k|G_j(x)|^2\leq 1 \ \forall x\in O\right\},
\label{chiapponi}
\end{align}
where $\partial_h^*\psi=\partial_h\psi-\psi \hat h$ for any smooth enough function $\psi$ and any $h\in QX^*$. 
If $O=X$ we denote $V_\gamma^F(f,X)$ by $V_\gamma^F(f)$.
Further, let $h\in  QX^*$. 
When $F={\rm span}\{h\}$, we denote $V_\gamma^F(f,O)$
by $V_\gamma^h(f,O)$.
As $V_\gamma(f,O)$, also $V_\gamma^{F}(f,O)$ and $V_\gamma^{h}(f,O)$ are lower semicontinuous respect to $f$ in the $L^1$ topology.

We also define the weak {gradient} of $f\in BV(O,\gamma)$ along $F$. Let $D_\gamma^Of=\sigma|D_\gamma^Of|$ be the polar decomposition of $D_\gamma ^Of$. 
We define the weak gradient of $f$ along $F$ by
\begin{align*}
 D_\gamma^{O,F} f:= \sum _{j=1}^k [ \sigma, h_{i_j}  ]_H h_{i_j} |D_\gamma^Of|.
\end{align*}

If $O=X$, we denote $V_\gamma^{h}(f,X)$ by 
$V^h_\gamma(f)$; 
it is not hard to see that this definition coincides with \cite[Definition 3.8]{AmbMirManPal10}, by arguing as in Lemma \ref{tavolo}.

\section{Equivalent characterizations of  \texorpdfstring{$BV$}{} on domains}
\label{main}

Let $O\subseteq X$ be an open set. The aim of this section is to prove that, analogously to the case $O=X$ (see \cite[Theorem 4.1]{AmbMirManPal10}) it is possible to characterize the space $BV(O,\gamma)$ in terms of \eqref{vista} and \eqref{fattoria}. To begin with, we state the main theorem of the paper. 
\begin{thm}
Let $f\in L(\log L)^{1/2}(O,\gamma)$. The following are equivalent:
\begin{enumerate}
\item $f\in BV(O,\gamma)$;
\item $V_\gamma(f,O)<+\infty$;
\item $L_\gamma(f,O)<+\infty$.
\end{enumerate}
Moreover, if one (and then all) of the previous holds true, then $|D_\gamma^Of|(O)=V_\gamma(f,O)=L_\gamma(f,O)$.
\label{lauree}
\end{thm}

Since the proof is rather long, for reader's convenience we split it into two different subsections. In the former we prove implication $(2)\Rightarrow (1)$, in the latter we show that $(1)\Rightarrow (3)\Rightarrow (2)$.

\subsection{ \texorpdfstring{$ (2)\Rightarrow (1)$}{}}
In this subsection we show that, if $f\in L({\rm log}L)^{1/2}(O,\gamma)$ has finite variation, then $f\in BV(O,\gamma)$ and $|D_\gamma^Of|(O)=V_f(f,O)$. 

The following result is a part of \cite[Theorem 4.1]{AmbMirManPal10}.
\begin{pro}\label{noteemilia}
Let $u\in L({\rm log}L)^{1/2}(X,\gamma)$. Then, the following are equivalent:
\begin{itemize}
\item[$(i)$] there exists $\mu\in\mathscr M(X,H)$ such that
\begin{align*}
\int_{X}u\partial_h^* G d\gamma=-\int_X G[\mu,h]_H, \qquad \forall G\in{\rm Lip}_c(X).
\end{align*}
\item[$(ii)$] $V_\gamma(u)<+\infty$.
\end{itemize}
\end{pro}
We remark that, arguing component by component in a way similar to the proof of \cite[Theorem 4.1]{AmbMirManPal10} and taking into account 
\cite[Definition 3.1]{AmbMirManPal10} and subsequent paragraphs we have the following result.
\begin{rmk}\label{noteemilia1}
Let $h\in QX^*$ and let $u\in L({\rm log}L)^{1/2}(X,\gamma)$. Then, the following are equivalent:
\begin{itemize}
\item[$(i)$] there exists $\mu_h\in\mathscr M(X)$ such that
\begin{align*}
\int_{X}u\partial_h^* G d\gamma=-\int_X G d \mu_h, \qquad \forall G\in{\rm Lip}_c(X).
\end{align*}
\item[$(ii)$] $V_\gamma^h(u)<+\infty$.
\end{itemize}
Further, if we consider $F\subseteq QX^*$ with $F={\rm span}\{h_1,\ldots,h_m\}$ such that $V_\gamma^F(u)<+\infty$, 
then there exists a measure $\mu_F:=(\mu_1,\ldots,\mu_m)\in\mathscr M(X,F)$ which satisfies
\begin{align*}
\int_Xu\partial_h^* Gd\gamma=-\int_XGd[\mu_F,h]_H, \quad G\in{\rm Lip}_c(X),
\end{align*}
for any $h\in F$ and 
\begin{align*}
|\mu_F|(X)=|(\mu_1,\ldots,\mu_m)|(X)\leq V_\gamma^F(u).
\end{align*}
Finally, if $\sup_{F\subseteq QX^*}V_\gamma^F(u)<+\infty$, then $u\in BV(X,\gamma)$.
\end{rmk}

\begin{rmk}
The results of \cite{AmbMirManPal10} can be used because the definitions of $BV$ and variations are equivalent to ours. Moreover, from \cite[Theorem 4.1]{AmbMirManPal10}
the statements in Proposition \ref{noteemilia1} hold true for $G\in \fcon_b^1(X)$. We can generalize to $G\in {\rm Lip}_c(X)$ by arguing as in Lemma \ref{tavolo}.
\end{rmk}

\begin{thm}
\label{cassettiera}
{Let $f\in L({\rm log}L)^{1/2}(O,\gamma)$ be such that $V_\gamma(f,O)<+\infty$. 
Then, $f\in BV(O,\gamma)$, i.e., there exists a unique $H$-valued measure on $\mathcal B(O)$, denoted by $D_\gamma^Of$, 
which satisfies
\begin{align}
\label{nanoparticelle}
\int_Of{\rm div}_\gamma^FGd\gamma=-\int_O[G,dD_\gamma^Of]_H, \quad G\in{\rm Lip}_c(O,F),
\end{align}
where $F\subseteq QX^*$ finite dimensional subspace. Moreover, $|D_\gamma^Of|(O)= V_\gamma(f,O)$ and
\begin{align*}
|D_\gamma^O f|(A)=
\sup\left\{\int_Af{\rm div}^F_\gamma Gd\gamma:F\subseteq  QX^*\ {\rm { fin. \ dim.}}, \  G\in {\rm Lip}_c(A,F), 
\ |G(x)|_F\leq 1 \ \forall x\in A\right\},
\end{align*}
for any open set $A\subseteq O$.}
\end{thm}
\begin{proof}

Let $g$ be a Lipschitz function such that 
${\rm dist}({\rm supp}\ g,O^c)>0$ and $\|g\|_\infty\leq1$. Then, $g\overline f\in L({\rm log}L)^{1/2}(X,\gamma)$ and $\sup_{F\subseteq QX^*}V^F_\gamma(g\bar f)<+\infty$. 
Recalling the concepts in Subsection \ref{Orlicz} it is easy to see that 
\begin{align*}
\int_XA_{1/2}(|g\overline f|)d\gamma
\leq &\int_XA_{1/2}(|\overline f|)d\gamma
=\int_OA_{1/2}(|f|)d\gamma<+\infty,
\end{align*}
since $f\in L({\rm log}L)^{1/2}(O,\gamma)$. Hence, $A_{1/2}(|g\overline f|)\in L^1(X,\gamma)$ and 
therefore $g\overline f\in L({\rm log}L)^{1/2}(X,\gamma)$. Further, for any $t>0$ we have
\begin{align*}
\int_XA_{1/2}(|(g\overline f|/t)d\gamma
\leq\int_OA_{1/2}(|f|/t)d\gamma,
\end{align*}
and so $\|g\overline f\|_{L({\rm log}L)^{1/2}(X,\gamma)}\leq \|f\|_{ L({\rm log}L)^{1/2}(O,\gamma)}$. 
Finally, for any $F\subseteq QX^*$ finite dimensional and any $G\in{\rm Lip}_c(X,F)$ with $|G(x)|_F\leq 1$ for any $x\in X$, 
the function $gG$ belongs 
to ${\rm Lip}_c(O,F)$ and $\|gG\|_\infty\leq 1$. 
Since ${\rm div}^F_\gamma(gG)=g{\rm div}^F_\gamma G+[G,\nabla_Hg]_H$, it follows that
\begin{align*}
\int_X(g\bar f){\rm div}_\gamma^F Gd\gamma
= & \int_X \bar f{\rm div}_\gamma^F(gG)d\gamma-\int_Xf[ G,\pi_F (\nabla_H g)]_Hd\gamma \\
\leq&  V_\gamma(f,O)+\|\nabla_Hg\|_{L^\infty(O,\gamma;H)}\|f\|_{L^1(O,\gamma)},
\end{align*}
since $g$ is Lipschitz. From Proposition \ref{noteemilia}, there exists a Borel $H$--valued 
measure $D_\gamma^{X}(g\bar f)$ such that
\begin{align}
\int_X(g\bar f)\partial_h\psi d\gamma=-\int_X\psi d[h,D_\gamma^{X}(g\bar f)] _H, \label{gfprodotto}
\end{align}
for any $\psi\in {\rm Lip}_c(X)$ and any $h\in QX^*$. Hence, $g\bar f\in BV(X,\gamma)$ by Definition \ref{polenta}.

Let $A\Subset O$ be an open set. For every $g\in {\rm Lip}_c (O)$ such that $g_{|A}\equiv1$, 
we have seen that $g\bar f\in L({\rm log}L)^{1/2}(X,\gamma)$
and $g\bar f\in BV(X,\gamma)$. We define 
\[
\nu_{A}(B):=D_{\gamma}^{X}(g\bar f)(B\cap A)
\]
for every $B\in\mathcal B(X)$. Clearly, $\nu_{A}$ is concentrated on $A$ and by Remark \ref{equality} we deduce that $\nu_{A}$ does not depend on the choice
of $g$.
Further, Corollary \ref{spaghetto} implies
\begin{align*}
|\nu_{A}|(O)=  |\nu_{A}|(A)
=  \sup\left\{\int_{A}\left[G,\sigma\right]_{H} d|\nu_{A}|:G\in\mbox{Lip}_{c}(A,F), \ F\subseteq QX^*\mbox{ fin. dim.},
|G| _{H}\leq1\right\}
\end{align*}
 where $\nu_{A}=\sigma|\nu_{A}|$. For any $G\in\mbox{Lip}_{c}(A,F)$ we have 
\[
\int_{A}\left[G,\sigma\right]_{H}\ d|\nu_{A}|=-\int_{X}fg\mbox{div}_{\gamma}G\ d\gamma=-\int_{X}f\mbox{div}_{\gamma}G\ d\gamma.
\]
Since $\mbox{Lip}_{c}(A,F)\subseteq\mbox{Lip}_{c}(O,F)$, we infer that $\int_{X}f\mbox{div}_{\gamma}G\leq V_{\gamma}(f,O)$. Then, $|\nu_{A}|(O)\leq V_{\gamma}(f,O)$.

We consider an increasing sequence $(A_n)$ of open sets such that $A_n\Subset O$ for any $n\in\N$, $\bigcup_{n\in\N}A_n=O$ and for any $B\Subset O$ there exists $\overline n\in\N$ such that $B\Subset A_{\overline n}$. A possible choice is $A_{n}:=\{x\in O:\mbox{dist}(x,O^{c})>n^{-1},\|x\|\leq n\}$ for any $n\in\N$. Since $(|\nu_{A_n}|(O))_{n\in \N}$ is an increasing bounded sequence, it admits limit $\elle\leq V_\gamma(f,O)$.

By the definition of $\nu_{A}$, it is clear that if $m<n$
then ${\nu_{A_n}}_{|A_m}={\nu_{A_{m}}}_{|A_{m}}$. It follows that 
\[
|\nu_{A_{n}}-\nu_{A_{m}}|(O)=|\nu_{A_{n}}|(A_n\backslash A_m)
=|\nu_{A_{n}}|(A_n)-|\nu_{A_{m}}|(A_m)=|\nu_{A_{n}}|(O)-|\nu_{A_{m}}|(O)
\]
{for $n>m$. Since $|\nu_{A_{n}}|(O)$ converges to $\elle$, the previous equation implies that $
(\nu_{A_{n}})$ is a Cauchy sequence in $\mathscr{M}(O,H)$, which is a Banach space
with norm 
\[
\|\mu\|:=|\mu|(O)
\]
(see e.g. \cite[Section I.5]{DieUhl77}) the 
discussion after Corollary 6). 
Therefore, $(\nu_{A_{n}})_{n\in \N}$ 
converges to a measure which we denote by $D_\gamma^Of=\sigma_{1}|D_\gamma^Of|$. Moreover, $|D_\gamma^Of|(O)=\elle$ and $D_\gamma^Of\mres A_{n}=\nu_{A_{n}}$ for every $n\in\N$.}

By Definition \ref{variation} there exists a sequence of functions $G_{n}$
such that,  for every $n\in\N$, $| G_{n}|_{H}\leq1$, $G_{n}\in\mbox{Lip}_{c}(O,F_{n})$
for some finite dimensional subspace $F_{n}\leq H$ and
\begin{align*}
 V_{\gamma}(f,O)=\lim_{n\rightarrow+\infty}\int_{O}f\mbox{div}_{\gamma}G_{n} d\gamma.
\end{align*}
The assumptions on $(A_n)_{n\in \N}$ imply that there exists an increasing sequence $(m_{n})_{n\in\N}\subseteq\N$
such that for any $n\in\N$ we have $G_{n}\in\mbox{Lip}_{c}(A_{m_n},F_{{n}})$. Therefore,
\[
V_{\gamma}(f,O)=\lim_{n\rightarrow+\infty}\int_{X}f\mbox{div}_{\gamma}G_{n} d\gamma\leq 
\lim_{n\rightarrow+\infty}|\nu_{A_{m_n}}|(O)=|D_\gamma^Of|(O)\leq V_{\gamma}(f,O),
\]
hence $|D_\gamma^Of|(O)\leq V_{\gamma}(f,O)$.

Finally, let $G\in\mbox{Lip}_{c}(O)$ and $h\in QX^*$. There exists
$n\in\N$ such that $G\in\mbox{Lip}_{c}(A_{n},F_{n})$ and 
\begin{align*}
\int_{O}\partial_{h}^{*}fG\ d\gamma
= & \int_{A_{n}}\partial_{h}^{*}fG\ d\gamma
=\int_{A_{n}}G d\left[h,\nu_{A_n}\right]_H
=\int_{A_{n}}G\left[h,\sigma_1\right]_H
d|D_\gamma^Of| 
= \int_{O} G \left[h,\sigma_{1}\right]_H\ d|D_\gamma^Of|.
\end{align*}
Therefore, by Definition \ref{polenta} it follows that $f$ is a function of bounded variation with weak
gradient $D_\gamma^Of$
which does not depend on the choice of $A_n$.

The second part of the statement follows from Corollary \ref{spaghetto}.
\end{proof}

Arguing as in Theorem \ref{cassettiera} it is possible to prove that if $f$ has finite variation $F\subseteq QX^*$ finite dimensional, 
then there exists a measure $D_\gamma^{O,F}f\in\mathscr M(O,F)$.
\begin{cor}\label{lasagna}
Let $f\in L({\rm log}L)^{1/2}(O,\gamma)$, let 
$F={\rm span}\{h_{1},\ldots,h_{k}\}$, where $h_i
\in  QX^*$ are orthonormal, and let us consider 
$V_\gamma^{F}(f,O)$ defined in \eqref{chiapponi}.
If $V_\gamma^{F}(f,O)<+\infty$, then there exists a measure $D_\gamma^{O,F}f\in\mathscr M(O,F)$ which enjoys
\begin{align*}
\int_O f{\rm div}_\gamma^F Gd\gamma=-\int_O [G,dD_\gamma^{O,F}f]_H, \quad G\in {\rm Lip}_c(O,F),
\end{align*}
$|D_\gamma^{O,F}f|=V_\gamma^F(f,O)$ and for any open set $A\subseteq O$ we have
\begin{align*}
|D_\gamma^{O,F}f|(A)=\sup\left\{\sum_{j=1}^k\int_Af\partial_{h_{j}}^* G_jd\gamma:G_j\in {\rm Lip}_c(A), \ j=1,\ldots,k,  \ \sum_{j=1}^k|G_j(x)|^2\leq 1 \ \forall x\in A\right\}.
\end{align*}
Moreover, let $h\in QX^*$:  
if $V_\gamma^h(O,f)<+\infty$, then there exists a finite measure $D_\gamma^{O,h}f\in\mathscr M(O)$ which satisfies
\begin{align*}
\int_O f \partial_h Gd\gamma=-\int_O G dD_\gamma^{O,h}f, \quad G\in {\rm Lip}_c(O),
\end{align*}
$|D_\gamma^{O,h}f|=V_\gamma^h(f,O)$ and for any open set $A\subseteq O$
\begin{align*}
|D_\gamma^{O,h}f|(A)=\sup\left\{\int_Af\partial_{h}^* Gd\gamma:G\in {\rm Lip}_c(A),  \ |G(x)|\leq 1 \ \forall x\in A\right\}.
\end{align*}
\end{cor}

\subsection{The implications  \texorpdfstring{$(1)\Rightarrow (3)\Rightarrow (2)$}{}}
In this subsection we prove the remaining implications of Theorem~\ref{lauree}. Before stating the result we are interested in, we provide a useful result involving the Ornstein-Uhlenbeck semigroup $(T_t)_{t\geq0}$ and the space $BV(X,\gamma)$.
\begin{lemma}
\label{esami}
Let $f\in BV(X,\gamma)$, let $h\in H$ and let $O\subseteq X$ be an open set such that $|D^h_\gamma f|(\partial O)=0$. Then, 
\[
\lim_{t\to 0}
\int_O |\partial_hT_tf|d\gamma =
|D_\gamma^{X,h}f|(O).
\]
\end{lemma}
\begin{proof}
We can apply \cite[Theorem 4.4]{AmbMirManPal10} because the definitions of variation are equivalent, therefore we know that 
$\int_X |\partial_hT_tf|d \gamma \rightarrow  |D_\gamma^{X,h}f|(X)$ as $t\rightarrow0^+$.
Further, \cite[Theorem 8.2.3]{Bog06} gives 
\begin{align*}
\limsup_{t\rightarrow0}
\int_C |\partial_hT_tf|d \gamma\leq |D_\gamma^{X,h}f|(C),
\end{align*}
for any closed set $C\subseteq X$. Therefore,
\begin{align*}
|D_\gamma^{X,h}f|(O)\leq\liminf_{t\rightarrow0}
\int_O |\partial_hT_tf|d\gamma 
\leq \limsup_{t\rightarrow0}
\int_{\overline{O}} |\partial_hT_tf|d\gamma
\leq |D_\gamma^{X,h}f|(\overline O)
= |D_\gamma^{X,h}f|(O),
\end{align*}
since $|D_\gamma^{X,h}f|(\partial O)=0$.
\end{proof}

\begin{rmk}
\label{implicazione_facile}
We easily deduce that $(1)\Rightarrow (2)$ i.e. that $f\in BV(O,\gamma)$ implies $V_\gamma(f,O)=|D_\gamma^Of|(O)$.
Indeed, let $f\in BV(O,\gamma)$ and let $F$ be a finite dimensional subspace of $QX^*$. Then for any $G\in {\rm Lip}_c(O,F)$, the integration by parts formula \eqref{pioggia} gives
\begin{align}
\int_Of({\rm div}_\gamma^FG) d\gamma
= &-\int_O [G,dD_\gamma^Of]_H
= -\int_O[G,\sigma]_H d|D_\gamma^Of|\leq |D_\gamma^Of|(O),
\label{jax}
\end{align}
where $D_\gamma^O f=\sigma |D_\gamma^O f|$ is the polar decomposition of $D_\gamma^Of$.

\noindent
Since $D_\gamma^Of$ is a finite Radon measure, taking the supremum with respect to $F$ and $G\in {\rm Lip}_c(O,F)$ with $|G(x)|_F\leq 1$ for any $x\in O$,
in both sides of \eqref{jax} and taking into account \eqref{corsa1}  in Corollary \ref{spaghetto} we infer that $V_\gamma(f,O)\leq |D_\gamma^Of|(O)$.
In particular, from Theorem \ref{cassettiera} we conclude that $V_\gamma(f,O)=|D_\gamma^Of|(O)$.
\end{rmk}

Implication $(1)\Rightarrow (3)$ is the content of the following proposition.
\begin{pro}
If $f\in BV(O,\gamma)$ then $L_\gamma(f,O)\leq |D_\gamma^Of|(O)$.
\label{cornice}
\end{pro}

\begin{proof}
We adapt the proofs in \cite[Theorem 3.9]{AmbFusPal00} and in \cite[Proposition 7.5.9]{tesiGiorgio}. 
Our aim is proving that there exists a sequence $(f_\varepsilon)\subseteq \mathbb D^{1,1}(X,\gamma)$ such that $f_\varepsilon\rightarrow f$ in $L^1(O,\gamma)$ and 
$\int_O |\nabla_Hf_\varepsilon|_Hd\gamma\rightarrow |D_\gamma^Of|(O)$ 
as $\varepsilon\rightarrow0$. 
Then by definition
$L_\gamma(f,O)\leq 
|D_\gamma ^Of|(O)$.

Assume that $f\in BV(O,\gamma)$. Since $|D_\gamma^Of|(O)<+\infty$, from Remark \ref{misura bordo} for all but at most countable $r\in(1,+\infty)$ we have
$|D_\gamma^Of|(\partial (O_{(-1/r)}))= 0$, where $O_{-\alpha}:=\{x\in O:{\rm dist}(x,O^c)>\alpha\}$ for any positive $\alpha$. Thus, for any $i\in\N$ we set
$O_i:=O_{-(r_i)^{-1}}$ such that $(r_i)$ is an increasing sequence of positive numbers, $r_i\rightarrow+\infty$ as $i\rightarrow+\infty$ and $|D_\gamma^Of|(\partial 
(O_{(-1/r_i)}))=0$ for any $i\in\N$.
We introduce a sequence of Lipschitz functions $(\varphi_i)_{i\in\N}\subseteq {\rm Lip}_c(O)$ such that $\varphi_i\equiv 1$ on $O_i$ for any $i\in\N$. Further, we define
\begin{align*}
f_i:=
\begin{cases}
f\varphi_i, & {\rm in} \ O, \\
0, & {\rm in} \ {O}^c.
\end{cases}
\end{align*}
Arguing as in the proof of {Theorem} \ref{cassettiera}, it follows that $f_i\in BV(X,\gamma)$ and 
\begin{align}
\label{der_i}
D^X_\gamma f_i=\varphi_i
D_\gamma^O f+f\nabla_H\varphi_i\gamma.
\end{align}
Since ${f_i}_{|O_i}\equiv f$, it follows that 
${D^X_\gamma f_i\mres O_i}= D^O_\gamma f\mres O_i$. Let us set
\begin{align*}
f_t:=T_t \overline{f}, \quad f_{i,t}:=T_t(f_i), \quad i\in\N, \ t \geq0,
\end{align*}
where $(T_t)_{t\geq0}$ is the Ornstein-Uhlenbeck semigroup given by
\eqref{bassano}. 
Since $\overline{f},f_i\in L(\log L)^{1/2}(X,\gamma)$, it follows that 
$f_{i,t}\in \mathbb D^{1,1}(X,\gamma)$ for any $i\in\N$ and any $t>0$ and $f_{i,t}\rightarrow f_i$ (resp. $f_t\rightarrow f$) 
in $L(\log L)^{1/2}(O_i,\gamma)$ (resp. $L(\log L)^{1/2}(O,\gamma)$) as $t\rightarrow0^+$ (see \cite[Proposition 3.6(i)-(ii)]{FuHi01}). 
Therefore $f_t\rightarrow f$ (resp. $f_{i,t}\rightarrow f_i$) in $L^{1}(O,\gamma)$ (resp. $L^{1}(O_i,\gamma)$) as $t\rightarrow0^+$. Hence, from Lemma \ref{esami} and \eqref{der_i} we deduce that
\begin{align}
\lim_{t\to 0}
\int_O |\nabla_Hf_{i,t}|_H d \gamma
=
 |D^X_\gamma f_{i}|(O), 
 \quad 
 \lim_{t\to 0}
\int_{O_i} |\nabla_Hf_{i,t}|_H d\gamma=
 |D^O_\gamma f|(O_i), \quad i\in \N.
 \label{irati}
\end{align} 

We define $U_1:=O_1$, $U_2:=O_2$ and $U_i:=O_{i}\setminus \overline{O_{i-2}}$ for any $i>2$.
Further, $(U_i)$ is an open covering of $O$ and $U_i\cap U_j=\emptyset$ for any $|i-j|>1$. Therefore, from \cite[Corollary 1.4]{AlbMaRoc97} with $W={\rm Lip}_c(O)$, there exists a partition of unity $(\psi_i)\subseteq {\rm Lip}_c(O)$ subordinated to $(U_i)$, i.e.,
\begin{align*}
\psi_i\geq0,  \quad \sum_{i\in\N}\psi_i\equiv1, \qquad {\rm supp}(\psi_i)\subseteq U_i, \ i\in\N.
\end{align*}
Let us fix $\varepsilon>0$. Since $D_\gamma^Of$ is a finite Radon measure and thanks to Lemma \ref{esami} there exists $i_\varepsilon>0$ such that
\begin{align}
|D_\gamma^Of|(O\setminus O_{i})|\leq \varepsilon, \qquad i\geq i_\varepsilon-1.
\label{didier}
\end{align}
Moreover, from the convergence of $f_{i,t}$ to $f$ in $L(\log L)^{1/2}(O_i,\gamma)$ and \eqref{irati} there exists $t_\varepsilon>0$ such that for any $i\in\{1,\ldots,i_\varepsilon\}$ we have
\begin{align}
& \|f_{i_\varepsilon,t_\varepsilon}-f\|_{L(\log L)^{1/2}(O_{i_\varepsilon},\gamma)}\leq \varepsilon, \label{martinez} \\
& \|f_{i_\varepsilon,t_\varepsilon}-f\|_{L^1(O_{i_\varepsilon},\gamma)}\leq2^{-i} (1+\|\nabla_H\psi_i\|_{L^\infty(X;H)})^{-1}\varepsilon, \label{daric} \\
& \left|\int_{O_{(i_\varepsilon-1)}}|\nabla_H f_{i_\varepsilon,t_\varepsilon}|_Hd\gamma-|D_\gamma ^Of|(O_{(i_\varepsilon-1)})\right|\leq \varepsilon \label{modric}.
\end{align}
Finally, again from the convergence of $f_{i,t}$ to $f$ in $L(\log L)^{1/2}(O_i,\gamma)$ and \eqref{irati}, for any $i>i_\varepsilon$ there exists $t_{\varepsilon,i}>0$ such that
\begin{align}
& \|f_{i,t_{\varepsilon,i}}-f\|_{L(\log L)^{1/2}(O_{i},\gamma)}\leq 2^{-i}\varepsilon, \label{southgate} \\
& \|f_{i,t_{\varepsilon,i}}-f\|_{L^1(O_{i},\gamma)}\leq 2^{-i}(1+\|\nabla_H\psi_i\|_{L^\infty(X;H)})^{-1}\varepsilon, \label{kane} \\
& \left| \int_{O}\psi_i|\nabla_Hf_{i,t_{\varepsilon,i}}|_Hd\gamma-\int_{O}\psi_id|D_\gamma^Of|\right|\leq 2^{-i}\varepsilon \label{hazard}.
\end{align}
We set
\begin{align}
\label{mbappe}
f_{\varepsilon,i}:=
\begin{cases}
f_{i_\varepsilon,t_\varepsilon}, & i\leq i_\varepsilon, \\
f_{i,t_{\varepsilon,i}}, & i>i_\varepsilon,
\end{cases} \qquad \qquad
f_\varepsilon:=\sum_{i\in\N}\psi_if_{\varepsilon,i}.
\end{align}
$f_\varepsilon$ is well defined since for any $x\in O$ the series is indeed a finite sum (the support of $\psi_i$ is contained in $U_i$ and therefore
the series in $f_\varepsilon(x)$ involves at most two terms).

By Lemma \ref{foglio}, $W^{1,1}(X,\gamma)= \mathbb D^{1,1}(X,\gamma)$. Since $f_{\varepsilon,i}\in \mathbb D^{1,1}(X,\gamma)$ we get  $f_{\varepsilon,i}\in W^{1,1}(X,\gamma)$.
It is clear that  {$f_\varepsilon\in D{\E}_{h}^{O}$} for all $h\in Q X^{*}$ and there exists $\nabla_{H}f_\varepsilon=\sum_{i\in\N}\psi_i\nabla_Hf_{\varepsilon,i}$ which satisfies $\partial_{h}f_\varepsilon= 
[\nabla_{H}f_\varepsilon,h]_H$. So, to prove that
$f_\varepsilon\in W^{1,1}(O,\gamma)$, it suffices to prove that $f_\varepsilon\in L^1(O,\gamma)$ and $\nabla_{H}f_\varepsilon\in L^1 (O,\gamma;H)$.

We stress that ${f_\varepsilon}_{|O^c}\equiv 0$ and $\nabla_H f_\varepsilon$ is well defined $\gamma$-a.e. $x\in X$. It is also worth noticing that ${f_\varepsilon}_{|O_{(i_\varepsilon-1)}}={f_{i_\varepsilon,t_\varepsilon}}_{|O_{(i_\varepsilon-1)}}$. Then,
\begin{align}
\int_O|f_\varepsilon|d\gamma
\leq & \int_O\left|\sum_{i\in\N}\psi_i(f_{\varepsilon,i}-f)\right|d\gamma
+\int_O|f|d\gamma \notag \\
\leq & \int_O\left|\sum_{i=1}^{i_\varepsilon}\psi_i(f_{\varepsilon,i}-f)\right|d\gamma+\int_O\left|\sum_{i=i_\varepsilon+1}^{+\infty}\psi_i(f_{\varepsilon,i}-f)\right|d\gamma
+\|f\|_{L(\log L)^{1/2}(O,\gamma)} \notag \\
\leq & 2\varepsilon+\|f\|_{L(\log L)^{1/2}(O,\gamma)}, \label{kazan}
\end{align}
thanks to \eqref{daric}, \eqref{kane}, \eqref{mbappe} and the fact that ${\rm supp}(\psi_i)\subseteq U_i\subseteq O_i$ for any $i\in\N$. Let us consider $\nabla_H f_\varepsilon$. We have
\begin{align*}
\int_O|\nabla_H f_\varepsilon|_Hd\gamma
\leq & \int_O\left| \sum_{i\in\N}\nabla_H\psi_i f_{\varepsilon,i}\right|_Hd\gamma
+\int_O\left| \sum_{i\in\N}\psi_i \nabla_H f_{\varepsilon,i}\right|_Hd\gamma=:I_1+I_2.
\end{align*}
Let us deal with $I_1$. Since $\sum_{i\in\N}\nabla_H\psi_i\equiv0$, it follows that
\begin{align}
\label{wroclaw}
I_1=\int_O\left| \sum_{i\in\N}\nabla_H\psi_i (f_{\varepsilon,i}-f)\right|_Hd\gamma\leq \varepsilon,
\end{align}
from the definition of $\psi_i$ and of $f_{\varepsilon,i}$, and by applying \eqref{daric} and \eqref{kane}. As far as $I_2$ is concerned, we get
\begin{align*}
I_2
\leq & \int_O \sum_{i=1}^{i_\varepsilon}\psi_i|\nabla_H f_{\varepsilon,i}|_Hd\gamma+\int_O \sum_{i=i_\varepsilon+1}^{+\infty}\psi_i|\nabla_H f_{\varepsilon,i}|_Hd\gamma
=:  J_1+J_2.
\end{align*}
We recall that from the definition $\sum_{i=1}^{i_\varepsilon}\psi_i\equiv1$ on $O_{(i_\varepsilon-1)}$ and from \eqref{mbappe} we have $f_{\varepsilon,i}=f_{i_\varepsilon,t_\varepsilon}$ for $i=1,\ldots,i_\varepsilon$. Hence, \eqref{modric} gives
\begin{align}
\label{mosca}
J_1
= & \int_{O_{(i_\varepsilon-1)}} |\nabla_H f_{i_\varepsilon,t_\varepsilon}|_Hd\gamma\leq |D_\gamma^Of|(O_{(i_\varepsilon-1)})+ \varepsilon.
\end{align}
As far as $J_2$ is concerned, we stress that $\psi_i\equiv0$ on $O_{(i_\varepsilon-1)}$ for any $i\geq i_\varepsilon+1$. Therefore, from \eqref{hazard} we deduce that
\begin{align}
\label{rabat}
J_2
\leq & \sum_{i\in\N}\int_{O\setminus O_{(i_\varepsilon-1)}}\psi_id|D_\gamma^Of|+\varepsilon=|D_\gamma^Of|(O\setminus O_{(i_\varepsilon-1)})+\varepsilon,
\end{align}
and the claim is so proved.

Finally, above computations reveal that $f_\varepsilon\rightarrow f$ in $L^1(O,\gamma)$ and 
$\int_O |\nabla_Hf_\varepsilon|_H d \gamma
\rightarrow |D_\gamma^Of|(O)$ as $\varepsilon\rightarrow0$. Indeed, \eqref{kazan} shows that
\begin{align*}
\int_O|f_\varepsilon-f|d\gamma\leq \varepsilon.
\end{align*}
Moreover, we have
\begin{align*}
\left|\int_O|\nabla_Hf_\varepsilon|_Hd\gamma-|D_\gamma^Of|(O)\right|
\leq & I_1+\Big|J_1-|D_\gamma^Of|(O)\Big|
+\left||D_\gamma^Of|(O_{(i_\varepsilon-1)})-|D_\gamma^Of|(O)|\right|+J_2.
\end{align*}
Therefore, from \eqref{didier}, \eqref{wroclaw}, \eqref{mosca} and \eqref{rabat} we conclude that
\begin{align*}
\left|\int_O|\nabla_Hf_\varepsilon|_Hd\gamma-|D_\gamma^Of|(O)\right|
\leq 5\varepsilon.
\end{align*}
\end{proof} 

The following proposition, which shows that $(3)\Rightarrow (2)$, concludes the proof of Theorem \ref{lauree}.
\begin{pro}
Let $f\in L(\log L)^{1/2}(O,\gamma)$. If $L_\gamma(f,O)<+\infty$ then $V_\gamma(f,O)\leq L_\gamma(f,O)$.
\end{pro}
\begin{proof}
 Let $f\in L(\log L)^{1/2}(O,\gamma)$ and let us show that $L_\gamma(f,O)<+\infty$ implies $V_\gamma(f,O)\leq  L_\gamma(f,O)$. Since $L_\gamma(f,O)<+\infty$, 
 there exists a sequence $(f_n)\subseteq  W^{1,1}(O,\gamma)$ such that $f_n\rightarrow f$ in $L^1(O,\gamma)$ as $n\rightarrow+\infty$ and 
 \[
 \lim_{n\to+\infty}
\int_O |\nabla_Hf_n|_H d\gamma = L_\gamma(f,O).
 \] 
Further, for any $G\in {\rm Lip}_c(O,F)$ where $F$ is a finite dimensional subspace $QX^*$ and such that $|G(x)|_H\leq 1$ for any $x\in O$, we have ${\rm div}_{\gamma}^F G\in L^\infty(X)$. 
Then thanks to Remark~\ref{lugaresi}, we get
\begin{align*}
\int_Of({\rm div}^F_\gamma G)d\gamma
= & \lim_{n\rightarrow+\infty}\int_Of_n({\rm div}^F_\gamma G)d\gamma
= \lim_{n\rightarrow+\infty}\left(-\int_O[\nabla_Hf_n,G]_Hd\gamma\right)
\leq L_\gamma(f,O).
\end{align*}
Taking the supremum over $G$, we get $V_\gamma(f,O)\leq L_\gamma(f,O)$.
\end{proof}

\section{Further results} 
\label{further}
In this section we collect some consequences of the results of Section \ref{main}. At first, we give a sufficient condition (related to the Ornstein-Uhlenbeck semigroup $(T_t)_{t\geq0}$ introduced in \eqref{bassano}) which ensures that $f\in L(\log L)^{1/2}(O,\gamma)$ belongs to $BV(O,\gamma)$. 
We stress that, differently from \cite[Theorem 4.1]{AmbMirManPal10}, we don't have the equivalence of this condition with those in Theorem \ref{lauree} 
since we are concerning with the semigroup $(T_t)_{t\geq0}$ defined on the whole space $X$. Unfortunately, at the best of our knowledges there is no good definition of 
the Ornstein--Uhlenbeck semigroup on open domains in Wiener spaces and therefore we don't recover the same result of \cite{AmbMirManPal10}. 
For any $f\in L(\log L)^{1/2}(O,\gamma)$ we define the (possible infinite) limit
\begin{align}
\label{butta}
\mathscr J(f,O):=\liminf_{t\downarrow0}\int_O| \nabla_HT_t(\bar f)|_Hd\gamma,
\end{align}
and we show that if $\mathscr J(f,O)<+\infty$, then $f\in BV(O,\gamma)$ and $|D_\gamma^Of|(O)\leq \mathscr I(f,O)$.
\begin{pro}
Let $f\in L(\log L)^{1/2}(O,\gamma)$ and let $\mathscr J(f,O)<+\infty$. Then, $f\in BV(O,\gamma)$ and $|D_\gamma^Of|(O)\leq \mathscr J(f,O)$.
\end{pro}
\begin{proof}
Arguing as in the proof of Theorem \ref{cassettiera} it follows that $\overline f\in L({\rm log}L)^{1/2}(X,\gamma)$ and therefore from \cite[Proposition 3.6$(i)$]{Fu00} we infer that $T_t\overline f\in \mathbb D^{1,1}(X,\gamma)$ for any $t>0$. Then, it is enough to apply Theorem \ref{lauree}(3) with $f_n:=T_{t_n}(\bar f)$, being $(t_n)$ be any sequence of positive real numbers which satisfies $t_n\downarrow0^+$ as $n\rightarrow+\infty$.
\end{proof}

Corollary \ref{lasagna} allows us to prove a connection between the variation $V_\gamma^{h}(f,O)$ along $h$ and the one dimensional sections of $f$ 
(see \cite[Chapter 3.11, Theorem 3.103]{AmbFusPal00} for the finite dimensional case and \cite[Theorem 3.10]{AmbMirManPal10} for the Wiener setting). To this aim, 
let us fix $h\in QX^*$ (so $\hat {h}\in X^*$) and let us set $K:={\rm ker} (\hat {h})$. For any open set $\Omega\subseteq X$ and for $\gamma$-a.e. $y\in K$ we introduce 
the set $\Omega_y^h:=\{t\in\R:y+th\in \Omega\}$.
Further, for any function $u:\Omega\rightarrow \R$ we set $u_y:\Omega_y^h\rightarrow \R$ as $u_y(t)=u(y+th)$.
Moreover, we consider the decomposition of $\gamma=\gamma_1\otimes \gamma^\perp_h$, where $\gamma_1=\gamma\circ\pi_h^{-1}$ and 
$\gamma_h^\perp=\gamma\circ(I-\pi_h)^{-1}$ and $\pi_h:X\rightarrow {\rm span}\{h\}$ is the projection on ${\rm span}\{h\}$. 
Finally, for any $f\in L({\rm Log}L)^{1/2}(\Omega,\gamma)$ we define
\begin{align*}
V_{\gamma_1}^h(f_y,\Omega_y^h):=\sup\left\{\int_{\Omega_y^h}f_y(t)(\psi'(t)-t\psi(t))d\gamma_1(t):
\psi\in{\rm Lip}_c(\Omega_y^h), \ |\psi(t)|\leq 1\ \forall t\in \Omega_y^h\right\}.
\end{align*}
We stress that $\psi'(t)-t\psi(t)={\rm div}_{\gamma_1}\psi(t)$ is 
the Gaussian divergence in dimension $1$. 

\begin{pro}
\label{comizio}
Let $O\subseteq X$ be an open set, let $f\in L({\rm log}L)^{1/2}(O,\gamma)$ and let $h\in QX^*$. Then,
\begin{align}
\label{montare}
V_\gamma^{h}(f,O)=\int_{K}V^h_{\gamma_1}(f_y,O_y^h)\gamma^\perp_h(dy).
\end{align}
\end{pro}

\begin{proof}
At first, we remark that the inequality $\leq$ in \eqref{montare} easily follows from Fubini's theorem. Indeed,
\begin{align*}
\int_Of(x)\partial_h^*Gd\gamma=
\int_K \int_{O_y^h}f_y(t)(G'(t)-tG(t))d\gamma_1(t)\gamma^\perp_h(dy)
\leq \int_KV_{\gamma_1}^h(f_y,O_y^h)\gamma^\perp_h(dy),
\end{align*}
for any admissible function $G$. Taking the supremum over $G$ we get the desired inequality.

In order to prove the converse one, we use both an approximation and a smoothing argument. Since $|D_\gamma^{O,h}f|(O)<+\infty$, 
from Remark \ref{misura bordo} there exists at most countably many indexes $\varepsilon\in(0,1)$ such that $|D_\gamma^{O,h}f|(\partial O_{-\varepsilon})\neq 0$. 

Let $\varepsilon\in(0,1)$ be such that $|D_\gamma^{O,h}f|(\partial O_{-\varepsilon})=0$, we introduce $g_{\varepsilon}$ such that $g_{\varepsilon} \in {\rm Lip}_c(O)$ and
$g_{\varepsilon|O_{-\varepsilon}} \equiv 1$.
For $t>0$ let use define $f_t:=T_t(g_{\varepsilon}\bar f)$ where $T_t$ is the Ornstein-Uhlenbeck semigroup in $L^1(X,\gamma)$, which is a strongly continuous semigroup, therefore
$f_{t|O_{-\varepsilon}}$ converges to $fg_{\varepsilon|O_{-\varepsilon}}$ in $L^1$ for $t\rightarrow 0$.
{$f_t$ is in $\mathbb{D}^{1,p}(X,\gamma)=W^{1,p}(X,\gamma)$, hence it is in $BV(X,\gamma)$. Arguing as in the first part of the proof of Theorem \ref{cassettiera},}
we have $g_{\varepsilon}\bar f \in BV(X,\gamma)$, and from Lemma \ref{esami} it follows that 
$|D_\gamma^{X,h} f_t|(O_{-\varepsilon})\rightarrow|D_\gamma^{X,h} (g_{\varepsilon}\bar f)|(O_{-\varepsilon})$ as $t\rightarrow0$. 
From Corollaries \ref{spaghetto} and  \ref{lasagna} we deduce that $|D_\gamma^{X,h}f_t|(O_{-\varepsilon})$ converges 
to $|D_\gamma^{O,h}f|(O_{-\varepsilon})=|D_\gamma^{X,h} (g_{\varepsilon}\bar f)|(O_{-\varepsilon})$ as $t\rightarrow0$. 
Moreover, since
\begin{align*}
\int_{K}\|(f_t)_y-(f)_y\|_{L^1((O_{-\varepsilon})_y^h,\gamma_1)}\gamma^\perp_h(dy)=\int_{O_{-\varepsilon}}|f_t-f|d\gamma\rightarrow0, \qquad t\rightarrow0,
\end{align*}
there exists a sequence $(t_n)$ decreasing to $0$ as $n\rightarrow+\infty$ such that
\begin{align*}
\int_{(O_{-\varepsilon})_y^h}|(f_{n})_y-(f)_y|d\gamma_1\rightarrow0, \qquad n\rightarrow+\infty,
\end{align*}
for $\gamma^\perp_h$-a.e. $y\in K$, where $f_n:=f_{t_n}$ for any $n\in\N$. Hence, the lower semicontinuity of $V^h_{\gamma_1}$, Fatou's Lemma, the convergence of $(|D_\gamma^{X,h} f_n|(O_{-\varepsilon}))$ and Corollary \ref{lasagna} imply that
\begin{align}
\int_{K}V_{\gamma_1}^h(f_y,(O_{-\varepsilon})_y^h)\gamma^\perp_h(dy)
\leq & \liminf_{n\rightarrow+\infty}\int_{K}V_{\gamma_1}^h((f_n)_y,(O_{-\varepsilon})_y^h)\gamma^\perp_h(dy)=\notag  \\
= & \liminf_{n\rightarrow+\infty} |D_\gamma^{X,h}f_n|(O_{-\varepsilon})
=  |D_\gamma^{O,h}f|(O_{-\varepsilon})\leq V_\gamma^{h}(f,O). 
\label{treno}
\end{align}
Letting $\varepsilon\rightarrow0$ in \eqref{treno} we conclude.
\end{proof}

\section{ \texorpdfstring{$BV$}{} functions on domains in Hilbert spaces}
\label{general_measure}
In this section we show that the arguments in the proof of Theorem \ref{cassettiera} allow us to prove a different characterization of $BV$ functions on open domains in Hilbert spaces with respect to more general probability measures. In particular, we consider the setting of \cite{DaPLun18}, and we recall the main definitions and results. Let $X$ be a separable Hilbert space with inner product $\langle\cdot,\cdot \rangle$, let $\nu$ be a Borel probability measure on $X$ and let $R\in\mathcal L(X)$ be such that the following hypothesis is satisfied.
\begin{hyp}
For any $z\in X$ there exists ${\mathit v}_z\in\bigcap_{1\leq p<\infty}L^p(X,\nu)$ such that
\begin{align}
\label{ingegneria}
\int_X\langle R D\varphi,z\rangle d\nu=\int_X\varphi {\it v}_zd\nu, \quad \varphi\in C^1_b(X).
\end{align}
Here, $D$ denotes the Fr\'echet derivative of $\varphi\in C^1_b(X)$.
\label{frutteto}
\end{hyp}
In particular, it follows that the map $z\mapsto {\it v}_z$ is closed from $X$ to $L^p(X,\nu)$ for any $p\geq1$, and therefore it is continuous.
Hence, there exists a positive constant $C_p$ such that
\begin{align*}
\|{\it v}_z\|_{L^p(X,\nu)}\leq C_p\|z\|, \quad z\in X, \ p\geq 1.
\end{align*}
Further, since a crucial tool of our investigation is the space of Lipschitz functions, we need an additional hypothesis.
\begin{hyp}
\label{cavovga}
For any $G\in {\rm Lip}(X)$ there exists a subset $N=N_G\subseteq X$ such that $\nu(N)=0$ and $G$ is G\^ateaux differentiable on $X\setminus N$.
\end{hyp}

From \cite[Proposition 2.3]{DaPLun18} we inherit the following result.
\begin{pro}
The operator $RD:D(RD):=C^1_b(X)\rightarrow L^p(X,\nu;X)$ is closable in $L^p(X,\nu)$ and we denote by $M_p$ and $\mathbb D^{1,p}(X,\nu)$ its closure and the domain of its closure, respectively. In particular, if $p>1$ for any $f\in \mathbb D^{1,p}(X,\nu)$ and $z\in X$ it holds that
\begin{align}
\label{finalsix}
\int_X\langle M_pf,z\rangle d\nu=\int_Xf {\it v}_zd\nu.
\end{align}
\end{pro}

For any $p<\infty$ we denote by $M_p^*:D(M_p^*)\subseteq L^{p'}(X,\nu;X)\rightarrow L^{p'}(X,\nu)$ the adjoint operator of $M_p$. 
Hence, for any $f\in \mathbb D^{1,p}(X,\nu)$ and any $F\in D(M_p^*)$ we have
\begin{align*}
\int_XM_p^*Ffd\nu=\int_X\langle M_pf,F\rangle d\nu.
\end{align*}
To simplify the notations, for functions $f\in \cup_{p\geq 1}\mathbb D^{1,p}(X,\nu)$ we set $Mf:=M_pf$ for any $p\geq1$, and for vector fields $F\in\cup_{p>1}D(M_p^*)$ 
we set $M^*F:=M_p^*F$ for any $p>1$.

If $f\in \mathbb D^{1,p}(X,\nu)$ and $\varphi\in {\rm Lip}_b(X)$, then $(f\varphi)\in \mathbb D^{1,p}(X,\nu)$ and $M_p(f\varphi)=\varphi M_pf+fRD\varphi$,
where $D$ is the G\^ateaux derivative of $\varphi$. Hence, if $p>1$ from \eqref{finalsix} it follows that
\begin{align}
\int_X \varphi\langle M_p f,z\rangle d\nu=-\int_Xf(\langle RD \varphi,z\rangle-\varphi{\it v}_z)d\nu, \quad z\in X.
\label{vernacoliere}
\end{align}
This yields that $M_p^* (\varphi)(z)=-\langle RD \varphi,z\rangle+\varphi{\it v}_z$.
We stress that for $p=1$ the right-hand side of \eqref{vernacoliere} is not meaningful, in general, since nothing ensures that $f{\it v}_z\in L^1(X,\nu)$.
However, formula \eqref{vernacoliere} is the starting point for the definition of $BV$ functions in $(X,\nu)$ given in \cite[Definition 3.1]{DaPLun18}.

\begin{defn}
Let $f\in L^1(X,\nu)$ be such that $f{\it v}_z\in L^1(X,\nu)$ for any $z\in X$. We say that $u\in BV(X,\nu)$ if there exists an $X$-valued Borel measure ${\it m}\in\mathscr M(X,X)$ such that, setting ${\it m}_z(B):=\langle {\it m}(B),z\rangle$ for any $z\in X$ and any $B\in \mathcal B(X)$, we have
\begin{align}
\label{BV_X_Hilbert}
\int_Xf(\langle RD\varphi,z\rangle -\varphi {\it v}_z)d\nu=-\int_X\varphi d{\it m}_z, \quad \varphi\in C^1_b(X).
\end{align}
We also denote $m$ by $D_\nu f$ to stress its dependence on $\nu$ and $f$.
\end{defn}

Further, we introduce the variation along $z\in X$ of $f\in L^1(X,\nu)$.
\begin{defn}
Let $z\in X$ and let $f\in L^1(X,\nu)$ such that $f{\it v}_z\in L^1(X,\nu)$. Then, we define the variation of $f$ along $z$ by
\begin{align*}
V_\nu^z(f):=\sup\left\{\int_Xf(\langle RD\varphi,z\rangle-\varphi{\it v}_z)d\nu:\varphi\in C^1_b(X), \|\varphi\|_\infty\leq 1\right\}.
\end{align*}

Assume that $f{\it v}_z\in L^1(X,\nu)$ for any $z\in X$. Then, we define the variation of $f$ by
\begin{align*}
V_\nu(f):=\sup\left\{\int_XfM^*\varphi d\nu:\varphi\in C^1_b(X,F), \ F\subseteq X {\textrm { fin. dim.}}, \ \|\varphi\|_\infty\leq 1\right\}.
\end{align*}
It is easy to see that, if $\{e_1,\ldots,e_k\}$ is an orthonormal basis of $F$, then 
\[
M^*\varphi=-\sum_{i=1}^k\langle RD\varphi,e_i\rangle-\varphi v_{e_i}
\] 
for any $\varphi \in C_b^1(X,F)$.
\end{defn}

Let $O\subseteq X$ be an open set. We provide the definition of $BV$ functions and of the variation of a function on $O$.

\begin{defn}
Let $f\in L^1(O,\nu)$ such that $f{\it v}_z\in L^1(O,\nu)$ for any $z\in X$. We say that $f\in BV(O,\nu)$ if there exists an $X$-valued Borel measure $\mu\in\mathscr M(O,X)$ such that, setting $\mu_z(B):=\langle{ \mu}(B),z\rangle$ for any $z\in X$ and any $B\in \mathcal B(O)$, we have
\begin{align}
\label{scrivania}
\int_Of(\langle R\nabla \varphi,z\rangle -\varphi {\it v}_z)d\nu=-
\int_O\varphi d{\mu}_z, \quad \varphi\in {\rm Lip}_c(O),
\end{align}
where $\nabla$ denotes the G\^ateaux derivative. We also denote $\mu$ by $D_\nu^Of$ to stress its dependence on $\nu,O$ and $f$.
\end{defn}

\begin{defn}
Let $z\in X$ and let $f\in L^1(O,\nu)$ such that $f{\it v}_z\in L^1(O,\nu)$. Then, we define the variation of $f$ along $z$ by
\begin{align*}
V_\nu^z(f,O):=\sup\left\{\int_Of(\langle R\nabla \varphi,z\rangle-\varphi{\it v}_z)d\nu:\varphi\in {\rm Lip}_c(O), \ \|\varphi\|_\infty\leq 1\right\}.
\end{align*}
assume that $f{\it v}_z\in L^1(O,\nu)$ for any $z\in X$. We define the variation of $f$ in $O$ by
\begin{align*}
V_\nu(f,O):=\sup\left\{\int_OfM^{*,F}\varphi d\nu:\varphi\in {\rm Lip}_c(O,F), \ F\subseteq X {\textrm { fin. dim.}}, \ \|\varphi\|_\infty\leq 1\right\},
\end{align*}
where 
\begin{align}
\label{div_fin_dim_hilbert}
M^{*,F} \varphi:=\sum_{j=1}^m \langle RD\varphi,e_j\rangle -\varphi {\it v}_{e_j},
\end{align}
and $\{e_1,\ldots,e_m\}$ is an orthonormal basis of $F$.
\end{defn}

\begin{lemma}
If $O=X$, then:
\begin{enumerate}
 \item \eqref{scrivania} implies \eqref{BV_X_Hilbert} for any $z\in X$ and $f\in L^p(X,\nu)$ with  $p\in[1,+\infty)$;
 
 \item \eqref{BV_X_Hilbert} implies  \eqref{scrivania} for any $z\in X$ and $f\in L^p(X,\nu)$ with  $p\in(1,+\infty)$.
\end{enumerate}

\end{lemma}
\begin{proof}
 
To prove that \eqref{scrivania} implies \eqref{BV_X_Hilbert}, it is enough to argue as in the first part of the proof of Lemma \ref{tavolo}.

Let  $f\in L^p(X,\nu)$ for some $p>1$, we prove that \eqref{BV_X_Hilbert} gives \eqref{scrivania}. Let $G\in{\rm Lip}_c(X)$, let $\{e_k:k\in\N\}$ be an orthonormal basis 
of $X$ and let $G\in{\rm Lip}_c(X)$. Similarly to what is done in the proof of \cite[Lemma 2.1$(ii)$]{DaPLun18}, for any $n,k\in\N$ we set
\begin{align*}
G_{n,k}(x):=\int_{\R^n}G\left(P_nx+\frac1k\sum_{j=1}^n\xi_je_j\right)\rho_n(\xi)d\xi,
\end{align*}
where $P_nx:=\sum_{j=1}^n\langle x,e_j\rangle e_j$ for any $x\in X$ and $\rho_n$ is any nonnegative smooth function supported in the unit ball of $\R^n$ such that
$\int_{\R^n}\rho_n(\xi)d\xi=1$. We infer that $G_{n,k}\in\mathcal FC_b^1(X)$ and
\begin{align*}
& \lim_{k\rightarrow+\infty}\lim_{n\rightarrow+\infty}G_{n,k}(x)=G(x), \ \forall x\in X, \quad \|G_{n,k}\|_\infty\leq \|G\|_\infty, \\ 
& \|D G_{n,k}\|_{L^\infty(X,X)}\leq \|\nabla G\|_{L^\infty(X,X)}, \quad n,k\in\N.
\end{align*}
Then, $(G_{n,k})$ is bounded in $\mathbb  D^{1,p'}(X,\nu)$, and it follows that, up to a subsequence, 
$RDG_{n,k}$ weakly converges to a function $\Psi \in L^{p'}(X,\nu;X)$. By \cite[Chapter IV, Theorem 1.1]{DieUhl77} $L^{p'}(X,\nu;X)$ is the dual of $L^{p}(X,\nu;X)$,
hence, for any $z\in X$ and any $g\in C_b^1(X)$ we have
\begin{align*}
\int_Xg\langle R\nabla G, z\rangle d\nu
= &  \int_X(\langle M_pg,z\rangle-g{\it v}_z)Gd\nu
= \lim_{k,n\rightarrow+\infty} \int_X(\langle M_pg,z\rangle-g{\it v}_z)G_{k,n}) d\nu \\
= & \lim_{k,n\rightarrow+\infty} \int_Xg\langle RDG_{k,n},z\rangle  d\nu
= \int_Xg\langle \Psi,z\rangle d\nu.
\end{align*}
Thanks to \cite[Lemma 2.1(ii)]{DaPLun18} it follows that
\begin{align*}
\int_Xg\langle R\nabla G, z\rangle d\nu
= \int_Xg\langle \Psi,z\rangle d\nu,
\end{align*}
holds true for any $g\in C_b(X)$. If we define the measures $\nu_1:=\langle R\nabla G, z\rangle \nu$ and $\nu_2:=\langle \Psi,z\rangle \nu$, above arguments imply that
\begin{align*}
\int_Xgd\nu_1= \int_Xgd\nu_2, \quad g\in C_b(X),
\end{align*}
which means that $\nu_1=\nu_2$. Therefore, 
\begin{align*}
\int_Xfd\nu_1= \int_Xf d\nu_2,
\end{align*}
and we deduce that, up to a subsequence,
\begin{align*}
\int_Xf(\langle RDG_{n,k},z\rangle-G_{n,k}{\it v}_z)d\nu\rightarrow
\int_Xf(\langle \Psi,z\rangle-G{\it v}_z)d\nu
\int_Xf(\langle R\nabla G,z\rangle-G{\it v}_z)d\nu.
\end{align*}
Repeating this arguments for any $z\in X$ we get the thesis.
\end{proof}

We want to prove the equivalence of $V_\nu(f,O)<+\infty$ and $f\in BV(O,\nu)$. As usual, by standard arguments (see e.g. Remark \ref{implicazione_facile}) we infer that $f\in BV(O,\nu)$ gives $V_\nu(f,O)<+\infty$ and $V_\nu(f,O)\leq |D_\nu^Of|(O)$.
To show that if $V(f,O)<+\infty$ then $f\in BV(O,\nu)$ and $|D_\nu^Of|(O)\leq V_\nu(f,O)$ we first prove that, if $V_\nu^z(f,O)<+\infty$, then there exists a Borel measure $\mu_z$ which satisfies \eqref{scrivania}.

\begin{pro}
\label{prop_es_mis_z}
Let $z\in X$ be such that $R^*z\neq0$ and let $f\in L^1(O,\nu)$ be such that $f{\it v}_z\in L^1(O,\nu)$ and $V_\nu^z(f,O)<+\infty$. Then, there exists a Borel measure $\mu_z$ such that \eqref{scrivania} is satisfied.
\end{pro}
\begin{proof}
Let $f\in L^1(O,\nu)$ be such that $f{\it v}_z\in L^1(O,\nu)$. Further, let $g\in {\rm Lip}_c(O)$ such that $\|g\|_\infty\leq 1$. Clearly, $\overline fg\in L^1(X,\nu)$ and $(\overline fg){\it v}_z\in L^1(X,\nu)$. Further, for any $G\in C^1_b(X)$ with $\|G\|_\infty\leq1$, we have $gG\in{\rm Lip}_c(O)$ and $\nabla (gG)=G \nabla g+gDG$. Hence,
\begin{align*}
\int_X(\overline f g)(\langle RD G,z\rangle-G{\it v}_z)d\nu
= & \int_X\overline f(\langle R\nabla  (Gg),z\rangle-(Gg){\it v}_z)d\nu-\int_X\overline f G\langle R\nabla g,z\rangle d\nu \\
\leq & V_\nu^z(f,O)+\|R\|_{\mathcal L(X)}\|\nabla g\|_\infty\|f\|_{L^1(X,\nu)}.
\end{align*}
Hence, $V_\nu^z(\overline fg)<+\infty$ and from \cite[Theorem 3.3]{DaPLun18} there exists a measure ${\it m}_z$ which satisfies \eqref{scrivania} and $|{\it m}_z|(X)=V_z(\overline f g)$. Arguing as in Theorem \ref{cassettiera}, it is easy to build a measure $\mu_z$ as limit for $n\rightarrow+\infty$ of ${\it m}_z\mres A_n$, where $(A_n)$ is a suitable increasing sequence of open subsets of $O$, such that \eqref{scrivania} is fulfilled.
\end{proof}

We are ready to state the main theorem of this section, which is a characterization of the space $BV(O,\nu)$ in terms of the variation $V_\nu$.
\begin{thm}
\label{giuggiole}
Let $f\in L^p(O,\nu)$ for some $p>1$. Then, $f\in BV(O,\nu)$ if and only if $V_\nu(f,O)<+\infty$. In this case, $|D_\nu^Of|(O)=V_\nu(f,O)$.
\end{thm}
\begin{proof}
The proof follows the arguments in \cite[Theorem 3.5]{DaPLun18}, hence we only give a brief sketch.

The fact that $f\in BV(O,\nu)$ implies $V_\nu(f,O)<+\infty$ is standard, hence we limit ourselves to prove the converse implication.

Let $f\in L^1(O,\nu)$ be such that $V_\nu(f,O)<+\infty$. Therefore, for any $z\in X$ we have $V_\nu^z(f,O)<+\infty$. Let us fix an orthonormal basis $\{e_n:n\in\N\}$ of $X$ and for any $n\in\N$ let us set $\mu_n:=\mu_{e_n}$, where $\mu_{e_n}$ is the measure constructed in Proposition \ref{prop_es_mis_z}. We claim that
\begin{align}
\label{vector_measure}
m(B):=\sum_{n=1}^{+\infty}\mu_n(B)e_n, \quad B\in \mathcal B(O),
\end{align}
is a well defined vector measure belonging to $\mathscr M(O,X)$ which satisfies \eqref{vernacoliere} and $|m|(X)\leq V_\nu(f,O)$. To this aim, for any $k\in\N$ we define the measure
\begin{align*}
M_k(B):=\sum_{n=1}^k\mu_n(B)e_n, \quad B\in \mathcal B(O).
\end{align*}
From \eqref{corsa} we have
\begin{align*}
|M_k|(O)
= & \sup\left\{\int_O\langle\varphi,dM_k\rangle: \varphi\in{\rm Lip}_c(O,X), \ |\varphi(x)|_X\leq 1\  \forall x\in O\right\} \\
= & \sup\left\{\int_O\langle\varphi,dM_k\rangle: \varphi\in{\rm Lip}_c(O,P_k(X)), \ |\varphi(x)|_X\leq 1 \ \forall x\in O\right\},
\end{align*}
where $P_k$ is the projection on the subspace generated by $\{e_1,\ldots,e_k\}$. Then, Proposition \ref{vernacoliere} gives
\begin{align*}
\int_O\langle\varphi,dM_k\rangle
= & \sum_{n=1}^k\int_O\varphi d\mu_n
= \sum_{n=1}^k\int_Of(\langle R\nabla \varphi,e_n\rangle-\varphi v_{e_n})d\nu
= \int_OfM^{*,P_k(X)}\varphi d\nu \leq V_\nu(f,O)\|\varphi\|_\infty,
\end{align*}
for any $\varphi\in {\rm Lip}_c(O,P_k(X))$, where $M^{*,P_k(X)}$ has been defined in \eqref{div_fin_dim_hilbert}. This means that $|M_k|(O)\leq V_\nu(f,O)$ for any $k\in\N$. Classical results in measure theory imply that the series $\sum_{n=1}^{+\infty}\mu_n(B)e_n$ converges for any $B\in\mathcal B(O)$, that $m\in\mathscr M(O,X)$ and that $|m|(O)\leq V_\nu(f,O)$.

Finally, the validity of \eqref{scrivania}  for $\varphi\in {\rm Lip}_c(O)$ follows arguing as in \cite{DaPLun18}.
\end{proof}

\subsection{Examples}
Here we provide some examples of measures which satisfy both Hypotheses \ref{frutteto} and \ref{cavovga}. 

\paragraph{{\it Weighted Gaussian measure.}}
We consider a weighted Gaussian measure
\begin{align*}
\nu(dx)=\frac{e^{-U(x)}}{\int_Xe^{-U}d\gamma}\gamma(dx),
\end{align*}
where $\gamma$ is a Gaussian measure on $X$ and $U$ satisfies the following assumptions.
\begin{hyp}
$U:X\rightarrow \R\cup\{+\infty\}$ is a proper lower semicontinuous convex function such that $U\in \mathbb D^{1,q}(X,\gamma)$ for any $q>1$. 
\label{cacomela}\end{hyp}
Here we take $R=Q^{1/2}$. Under Hypothesis \ref{cacomela} it easy to see that Hypothesis \ref{frutteto} is satisfied, for any $z\in X$ we have ${\it v}_z=\hat h+\langle Q^{1/2}DU,z\rangle$, where $h=Q^{1/2}z$ and ${\it v}_z\in L^q(X,\nu)$ for any $z\in X$ and any $q\geq1$ (for a deep study of Banach space endowed with weighted Gaussian measure see \cite{Fer15}). Further, Hypothesis \ref{cavovga} is fulfilled since $\nu$ is absolutely continuous with respect to the Gaussian measure $\gamma$ and \cite[Theorem 5.11.2$(ii)$]{Bog98} implies that any Lipschitz function is G\^ateaux differentiable $\gamma$-a.e. in $X$. Hence, Theorem \ref{giuggiole} reads as follows.
\begin{thm}
Let $O\subseteq X$ be an open subset of $X$ and let $p>1$. Then, $f\in L^p(X,\nu)$ belongs to $BV(O,\nu)$ if and only if $V_\nu(f,O)<+\infty$. In this case for any $\varphi\in {\rm Lip}_c(O)$ we have
\begin{align*}
\int_Of(\langle Q^{1/2}D\varphi,z\rangle -\varphi  \hat h-\varphi\langle Q^{1/2}DU,z\rangle)d\nu=
-\int_X\varphi dm_z,
\end{align*}
for any $z\in X$, where $h=Q^{1/2} z$.
\end{thm}

\paragraph{{\it A non Gaussian product measure}.}
From \cite[Subsection 5.2]{DaPLun18} and \cite{DaPLuTu18} we can consider a non Gaussian example of measure $\nu$ which satisfies our assumptions. For $m\geq 1$ and $\mu>0$ we consider the measure
\begin{align*}
\nu_\mu(d\xi):=a\mu^{-\frac1{2m}}e^{-\frac{|\xi|^{2m}}{2m}}d\xi, \quad \xi\in \R,
\end{align*}
where $a:=(2m)^{1-1/(2m)}/\Gamma(1/(2m))$ is a normalizing factor such that $\nu_\mu(\R)=1$. We consider a sequence of positive numbers $\mu_j$, $j\in\N$, such that
\begin{align*}
\sum_{j\in\N}\mu_j^{\frac1m}<+\infty,
\end{align*}
which gives that the product measure on $\R^{\N}$ defined by
\begin{align*}
\nu:=\prod_{j\in\N}\nu_{\mu_j},
\end{align*}
is well defined and concentrated on $\ell^2$ (space of sequences with Euclidean norm). We set $X=L^2(0,1)$, we fix an orthonormal basis $\{e_k:k\in\N\}$ of $X$ consisting of equibounded functions and
we consider the standard isomorphism from $X$ to $\R^{\N}$, $x\mapsto (x_k)$, where $x_k=\langle x,e_k\rangle$ for any $k\in\N$. The induced measure is still called $\nu$,
and in \cite{DaPLuTu18} it has been proved that Hypothesis \ref{frutteto} is satisfied with $R=Q^{1/2}$, where $Q$ is the covariance operator of $\nu$, i.e.,
\begin{align*}
Qe_j=b_1\mu_j^{\frac1m}e_j, \quad j\in\N, \quad b_1=(2m)^{\frac1m}\frac{\Gamma\left(\frac3{2m}\right)}{\Gamma\left(\frac1{2m}\right)}.
\end{align*}
Finally, we show that the measure $\nu$ enjoys the property in Hypothesis \ref{cavovga}. Let $G\in {\rm Lip}(X)$ and, for any $n\in\N$ and let us consider
the $n$-dimensional subspace $X_n$ of $X$ generated by $\{e_1,\ldots,e_n\}$ and its orthogonal complement $X_n^{\perp}$. For any $y\in X_n^{\perp}$,
on the finite dimensional affine spaces $y+X_n$ we can choose the product measure
\begin{align*}
\nu^n:=\prod_{i=1}^n\nu_{\mu_i},
\end{align*}
which is absolutely continuous with respect to the Lebesgue measure in $\R^n$. Therefore, thanks to the finite dimensional Rademacher Theorem we infer that 
$D_{e_k}G$ exists $\nu$-a.e. for any $k=1,\ldots,n$, where $D_{e_k}$ denotes the G\^ateaux derivative along $e_k$. We conclude by proceeding as in \cite[Theorem 5.11.2$(ii)$]{Bog98}.

\bibliographystyle{plain}

\end{document}